\documentclass[oneside,english, 12pt]{amsart}

\usepackage{txfonts}
\usepackage{amsfonts}
\usepackage{lmodern}

\usepackage[T1]{fontenc}
\usepackage[latin9]{inputenc}
\usepackage{geometry}
\usepackage{amsthm}
\usepackage{amssymb}
\usepackage{amssymb,amsmath,amsthm,tikz}
\usetikzlibrary{matrix,arrows,decorations.pathmorphing}
\usepackage{enumerate}
\usepackage{url}

\usepackage[bookmarks=true, bookmarksopen=true,%
bookmarksdepth=3,bookmarksopenlevel=2,%
colorlinks=true,%
linkcolor=blue,%
citecolor=blue,%
filecolor=blue,%
menucolor=blue,%
urlcolor=]{hyperref}

\title{Simple sheaves for knot conormals}
\author{Honghao Gao}

\address{Department of Mathematics \\ Michigan State University \\ 619 Red Cedar Road \\ East Lansing \\ MI 48824 \\ USA}
\email{gaohongh@msu.edu}

\makeatletter

\numberwithin{equation}{section}
\numberwithin{figure}{section}

\theoremstyle{plain}
\newtheorem{thm}{Theorem}[section]
\newtheorem{lem}[thm]{Lemma}
\newtheorem{cor}[thm]{Corollary}
\newtheorem{prop}[thm]{Proposition}

\theoremstyle{definition}
\newtheorem{defn}[thm]{Definition}
\newtheorem{eg}[thm]{Example}

\newtheorem{notation}[thm]{Notation}

\theoremstyle{remark}
\newtheorem{rmk}[thm]{Remark}

\newcommand{\bbC}{{\mathbb{C}}}

\newcommand{\bbR}{{\mathbb{R}}}

\newcommand{\bbZ}{{\mathbb{Z}}}

\newcommand{\cA}{{\mathcal{A}}}

\newcommand{\cE}{{\mathcal{E}}}
\newcommand{\cF}{{\mathcal{F}}}
\newcommand{\cG}{{\mathcal{G}}}
\newcommand{\cH}{{\mathcal{H}}}

\newcommand{\cK}{{\mathcal{K}}}
\newcommand{\cL}{{\mathcal{L}}}
\newcommand{\cM}{{\mathcal{M}}}

\newcommand{\cP}{{\mathcal{P}}}

\newcommand{\cS}{{\mathcal{S}}}

\newcommand{\del}{{\partial}}

\newcommand{\la}{{\langle}}
\newcommand{\ra}{{\rangle}}

\tikzset{node distance=1.5cm, auto}

\usepackage{babel}

\begin{document}
\maketitle

\begin{abstract}
We classify simple sheaves microsupported along the conormal bundle of a knot. We also establish a correspondence between simple sheaves up to local systems and augmentations, explaining the underlying reason why knot contact homology representations detect augmentations.
\end{abstract}


\section{Introduction}

Given a knot in Euclidean three space or the three dimensional sphere, its conormal bundle is a conic Lagrangian subspace in the cotangent bundle of the ambient space, which is canonically a symplectic manifold. Using microlocal sheaf theory, one can study the subcategory of sheaves in the ambient space whose singular support is contained in the conormal bundle of the knot. Following a result of Guillermou-Kashiwara-Schapira \cite{GKS}, the dg derived category of such sheaves is a homogeneous Hamiltonian isotopy invariant of the knot conormal, and hence an isotopy invariant of the knot -- a knot invariant in short. Our first result studies a variant version, the category of simple abelian sheaves.

\begin{thm}[Theorem \ref{Sheafclassify}]\label{MainTheorem1}
For $X = \bbR^3$ or $S^3$, we classify the objects in $Mod^s_{\Lambda_K}(X)$, the simple abelian sheaves microsupported along the conormal bundle $\Lambda_K$ of the knot $K\subset X$.
\end{thm}

The microlocal sheaf theory we use in this paper mostly follows from the founding work of Kashiwara-Schapira \cite{KS}. The term ``microlocal'' refers to studying properties in the cotangent bundle, where the symplectic and contact geometry come in. The singular support, a key concept in microlocal sheaf theory, respects the dilation action along the fibre. In our setting, it is the knot conormal. The cotangent bundle removing the zero section can be dehomogenized to a contact manifold, and consequently the knot conormal becomes a Legendrian. The microlocal sheaf category is an invariant for the Legendrian knot conormal.

The knot contact homology is another invariant of the Legendrian knot conormal, which uses the theory of the $J$-holomorphic curves. Transforming the cosphere bundle into the one-jet space of sphere and identify the Legendrian knot conormal as a submanifold in the jet space, one can define a differential graded algebra which is as well a Legendrian isotopy invariant. The combinatorial version was first formulated by Ng \cite{Ng1, Ng2, Ng3}. The Floer theoretical version was introduced by Ekholm-Etnyre-Sullivan \cite{EES1, EES2, EES3}. These two versions are proven to be equivalent later by the four authors \cite{EENS1, EENS2}.

Augmentations, which originated from linearizing the dga to obtain the linearized contact homology, turn out to be more computable invariants of the knot. An augmentation of a dga is a morphism to a trivial dga. The definition is algebraic in general and we apply it to the Legendrian dga. When the Legendrian emerges from the conormal bundle of the knot, it is expected that some contact topological properties can be captured by the topology of the base ambient manifold. It was first formulated by Ng \cite{Ng3}, and later proven by Cornwell \cite{Cor13a, Cor13b}, that the KCH representation -- a type of the representation of the knot group -- detects a subset of augmentations.

We hope to use the sheaf theory to unwrap the somewhat mysteriously defined KCH representation and explain the reason why these representations detect augmentations. To each simple sheaf, we are able to define an associated augmentation (see Theorem \ref{sheaftoaug}). Further we show
\begin{thm}[Theorem \ref{CornNgfactor}]The map from KCH representations to augmentations studied by Ng and Cornwell factors through the following diagram.
$$\{\textrm{KCH Representations}\} \hookrightarrow \{\textrm{Simple abelian sheaves}\} \twoheadrightarrow \{\textrm{Augmentations}\}.$$

Moreover, there is a bijection between simple sheaves up to local systems and augmentations. The correspondence is summarized in the following table:
\smallskip
\begin{center}
\begin{tabular}{| r | l |}
\hline
\rule{0pt}{2.3ex}
simple sheaves up to local systems & augmentations $\epsilon$ \\[0.045cm]
\hline
\rule{0pt}{2.3ex}
irreducible KCH representations &  $\epsilon([e])\neq 0$ \\
\;\, irreducible unipotent KCH representations & $\epsilon([e])= 0$ but $\epsilon([\gamma])\neq 0$ for some $\gamma\in \pi_K$ \;\\
rank $1$ local systems on the knot & $\epsilon([\gamma])= 0$ for all $\gamma \in \pi_K$ \\[0.045cm]
\hline
\end{tabular}
\end{center}
Here $\pi_K$ is the fundamental group of the knot complement, and $e\in \pi_K$ is the identity.
\end{thm}

For a concise presentation, we introduce the notion of the unipotent KCH representation (in Section \ref{Sec:UniKCH}) and study its connection to augmentations (in Section \ref{Sec:UniKCHAug}).

The correspondence makes better sense if we restrict our attention to simple sheaves up to local systems. With this consideration, we are able to describe microlocally simple sheaves in the (dg) derived category of sheaves (see Proposition \ref{DerivedSheafloc}), which is the category studied by Guillermou-Kashiwara-Schapira. 

It becomes evident from the table that KCH representations consist of a subset of simple sheaves and therefore detect some augmentations. Another geometric interpretation can be found in the work of Aganagic-Ekholm-Ng-Vafa \cite{AENV} or the work of Cieliebak-Ekholm-Latschev-Ng \cite{CELN}. Briefly speaking, some of the $J$-holomorphic curves can be stretched close to the zero section of the ambient three dimensional sphere, whose boundary data are recorded by the knot group.

\smallskip
We continue with explanations on the overall theory.

It is no coincidence that simple sheaves are connected to augmentations. Augmentation have a functorial nature. Though defined algebraically, augmentations sometimes have geometric counterparts being exact Lagrangian fillings, with heuristics from the symplectic field theory \cite{El, EGH}. It is proven that an exact Lagrangian cobordism between two Legendrian knots induces a morphism of the associated dgas \cite{EHK}, which further induces a map between the sets of augmentations, via pullback.

Even better, the set of augmentations admits the structure of an $A_\infty$-category, which is in some sense a perturbed dg category with higher morphisms. We have to remind the reader that there can be more than one of such categorical structures \cite{NRSSZ, BC}. In a Fukaya-categorial point of view, augmentations arising from exact Lagrangian fillings can be regarded as objects in the infinitesimal Fukaya category, and their hom spaces inherit the $A_\infty$-structure from the Fukaya category, which depend on a choice of the perturbation data \cite{NRSSZ}.

The Nadler-Zaslow correspondence models the Lagrangian branes in the Fukaya category by microlocal sheaves \cite{Na, NZ}. In the case of Legendrian knots in the Euclidean threefold with the standard contact structure, it was first conjectured in \cite{STZ}, and later proven in \cite{NRSSZ}, that the counterparts of the augmentations are microlocally simple sheaves. Heuristically, augmentations are rank one representations of the Legendrian dga, which correspond to simple sheaves in the sheaf world. Yet in higher dimensions, such statements have not been established.

Hopefully we have explained why the sheaf theory is a tool to study knots. In fact, it is a powerful tool. It is proven by Shende \cite{Sh2}, using sheaf theory, and Ekholm-Ng-Shende \cite{EkNgSh}, using Floer homology, that knot conormals give complete knot invariants. In this paper, we restrict our attention to simple sheaves and their connections to other knot invariants -- augmentations and KCH representations. As a consequence of the theorem, we exhibit at the level of objects the correspondence between augmentations and sheaves.

There is a subtlety on the geometric set up. The ambient space where the Nadler-Zaslow interpretation works is different from the ambient space we consider in this paper, especially Theorem \ref{MainTheorem1}. The underlining geometric transform admits a sheaf quantization \cite{Ga2}. More explanations on the relations among these works can be found in \cite{Ga1}.

We also mention the work of Rutherford and Sullivan \cite{RS1, RS2, RS3} which localizes the dga of a general Legendrian surface. This work potentially establishes the foundation for studying the correspondence for an arbitrary Legendrian surface. However, we take a different approach in this paper.

\smallskip
The organization of the paper is as follows.

Section \ref{SectionKnotgroup} introduces topological concepts facilitating the presentation of the classification theorem. In Section \ref{Sec:Knotgroup} -- \ref{Sec:KCHRep}, we review the knot group and the KCH representation in literature. In Section \ref{Sec:UniKCH}, we define the unipotent KCH representation.

Section \ref{SectionSheaves} focuses on microlocal sheaves. After a quick introduction in Sections \ref{Sec:Ssupp} -- \ref{Sec:Simple}, we classify the simple abelian sheaves microsupported along the knot conormal in Section \ref{Sec:Classification}. In Sections \ref{Sec:Moduli} -- \ref{Sec:DerivedSheaves}, we study the moduli set of sheaves up to local systems, in both the abelian and derived settings.

Section \ref{SectionAugs} establishes the correspondence between simple sheaves and augmentations. In Section \ref{Sec:Augs}, we review the definition of the augmentation, and then define a map which sends a simple sheaf to an augmentation. In Section \ref{Sec:KCHAug}, we   show that the map from KCH representations to augmentations is compatible with our earlier definition, when the sheaf emerges from a KCH representation.  In Section \ref{Sec:UniKCHAug}, we study thoroughly the interplay between the unipotent KCH representations and augmentations. In Section \ref{Sec:SheafAug}, we establish the sheaf-augmentation correspondence, which is the second main theorem of the paper. The final Section \ref{Sec:AugPoly} is an application on augmentation polynomials.

\begin{notation}
Throughout the paper, we fix the following notations.
\begin{itemize}
\item
Let $X = S^3$ or $\bbR^3$. 
\item
Let $K\subset X$ be an oriented knot. We do not discuss links.
\item
Let $i:K \rightarrow X$ be the closed embedding of the knot. Let $j: X\setminus K\rightarrow X$ be the open embedding of the knot complement.

\item
Let $n(K)$ be a small tubular neighborhood of $K$. Its boundary $\del n(K)$ is a torus.

\item
Fix a ground field $k$. It is the field over which the representations, the sheaves, and the augmentations are defined, (but not the dga).
\end{itemize}
\end{notation}

\smallskip
\noindent\textbf{Acknowledgements.} We would like to thank Eric Zaslow for initiating the problem and advising. We thank St\'ephane Guillermou, as well as the referee, for important comments. We also thank Xin Jin, Lenhard Ng, Dmitry Tamarkin for helpful discussions. The paper is based on and extends the results in the author's PhD dissertation at Northwestern University. This work is partially supported by the ANR projection ANR-15-CE40-0007 ``MICROLOCAL''.

\section{Knot group and its representations}\label{SectionKnotgroup}

\subsection{Knot group}\label{Sec:Knotgroup}

The \textit{knot group} $\pi_K: =\pi_1(X\setminus K)$ is the fundamental group of the knot complement. The group is the same for both $X = S^3$ and $X = \bbR^3$. The knot group is a knot invariant. A \textit{meridian} is the boundary of an oriented disk which intersects transversely with the knot $K$ at a single point. A knot group $\pi_K$ has the following properties:
\begin{enumerate}
\item
$\pi_K$ is finitely generated and finitely presented; 
\item
$\pi_K$ can be generated by the meridians of $K$;
\item
any two meridians are conjugate to each other in $\pi_K$.
\end{enumerate}
The abelianization of $\pi_K$, or $H_1(X\setminus K)$, is isomorphic to $\bbZ$. A generator is represented by the class of any meridian.

A \textit{Seifert surface} $S\subset X$ of an oriented knot $K$ is an oriented surface whose boundary is $K$. Every knot $K$ admits a Seifert surface. A Seifert surface is not unique, but its relative homology class in $H_2(X,K)$ is unique. Following the long exact sequence of the relative pair $(X,K)$, we have an exact sequence
$$0\rightarrow H_2(X,K) \xrightarrow{\del} H_1(K)\rightarrow 0.$$
Clearly $H_1(K) \cong \bbZ$. The relative class of a Seifert surface is given by the preimage of $[K]$ under the boundary map $\del$.

The tubular neighborhood $n(K)$ has a torus boundary. A \textit{longitude} $\ell$ is the intersection of $\del n(K)$ with a Seifert surface $S$. It inherits a natural orientation from $K$. Since $[S]$ is unique, the homology class of $[\ell]$ in $H_1(T)$ is unique, and in this sense, the longitude is also unique.

The fundamental group $\pi_1(T)$ is abelian and isomorphic to $\bbZ\times \bbZ$. The closed embedding of $T$ into $X\setminus K$ induces a map $\pi_1(T)\rightarrow \pi_K$. The torus singles out a preferred meridian. The longitude commutes with the preferred meridians in $\pi_K$. The longitude is contractible only when $K$ is the unknot. Also note that every representation of $\pi_K$ induces a representation of $\pi_1(T)$ by composition.

\subsection{KCH representation}\label{Sec:KCHRep}

We first review the definition of the KCH representation \cite{Ng3, Cor13a, Cor13b}. The name ``KCH'' is an abbreviation of the ``knot contact homology''. We postpone to explain the relation between these representations and the knot contact homology after we have introduced augmentations.

\begin{defn}
Suppose $V$ is a vector space. Let $m\in\pi_K$ be a fixed meridian. A representation $\rho: \pi_K \rightarrow GL(V)$ is a \textit{KCH representation} if $\rho(m)$ is diagonalizable, and acts on $V$ as identity on a codimension $1$ subspace. In particular, the codimension constraint requires $\rho(m)\neq \textrm{id}_V$, namely $\rho(m)$ is not the identity on the complement of that codimension $1$ subspace.

A KCH representation is \textit{irreducible} if it is irreducible as a representation.
\end{defn}

\begin{rmk}\label{KCHRmk}

Because all meridians are conjugate to each other in $\pi_K$, for any meridian $m$,  the $\rho(m)$-action on $V$ has an invariant subspace of codimension $1$.

Since the preferred meridian $m$ and the longitude $\ell$ commute, their action matrices can be simultaneously diagonalized up to Jordan blocks. Therefore, there is a basis of $V$ under which we have,
$$\rho(m) = \begin{pmatrix} \mu_0 & \\ & I_{n-1} \end{pmatrix},
\quad
\rho(\ell) = \begin{pmatrix} \lambda_0 & \\ &\ast_{n-1} \end{pmatrix}.
$$
where $\mu_0 \neq 1$, $n =\dim V$, and $\ast_{n-1}$ is a square matrix of size $n-1$. An eigenvector of the eigenvalue $\mu_0$ in $\rho(m)$ is also an eigenvector of $\rho(\ell)$, corresponding to the eigenvalue $\lambda_0$. We do not impose any constraints on $\lambda_0$, but it is non-zero by construction.
\end{rmk}

Next we show that a KCH representation is always an extension between an irreducible KCH representation and a trivial representation.

Let $\{m_i\}_{i\in I}$ be a finite set of meridian generators of $\pi_K$. Let $(\rho,V)$ be a KCH representation, and suppose $v_i$ is a distinguished eigenvector of $\rho(m_i)$, i.e. $\rho(m_i)\,v_i = \mu_0 \,v_i$. Define the \textit{meridian subspace} $V_0\subset V$ by 
\begin{equation}\label{Merdiansubspace1}
V_0:= \textrm{Span}_k\{v_I\}.
\end{equation}
The meridian subspace has the following properties:
\begin{enumerate}
\item
$\textrm{Span}_k\{v_I\}$ is $\pi_K$ invariant, i.e. $\textrm{Span}_k\{v_I\} = \textrm{Span}_{k[\pi_K]}\{v_I\}$, (\cite[Lemma 3.10]{Cor13b});
\item
$\textrm{Span}_{k[\pi_K]}\{v_I\} = \textrm{Span}_{k[\pi_K]}\{v_i\}$ for any $i\in I$, (if $m_j = g^{-1}m_ig$, then $v_i= \rho(g)v_j$).
\end{enumerate}
It shows that $V_0$ is a sub-representation by (1), and is irreducible by (2).

\begin{lem}\label{quottrivial1}
Suppose $\rho:\pi_K\rightarrow GL(V)$ is a KCH representation, then the quotient representation $\bar{\rho}: \pi_K\rightarrow GL(V/V_0)$ is trivial.
\end{lem}
\begin{proof}
To prove the quotient representation is trivial, it suffices to show that each generator acts on the quotient vector space as identity. Because the generators are conjugate to each other in the knot group, it suffices to prove for one generator. Let's consider $\rho(m_1)$. We can find a basis of $V$ including $v_1$ such that $\rho(m_1)$ acting on all other basis vectors as identity. Since $v_1\in V_0$, we have $\bar{\rho}(m_1) = \textrm{id}_{V/V_0}$. The proof is complete.

\end{proof}

Suppose $Y$ is a manifold, and $\pi_1(Y)$ is its fundamental group. It is well known the equivalence between the category of $\pi_1(Y)$ representations and the category of local systems on $Y$:
\begin{equation}\label{rep-loc}
Rep(\pi_1(Y)) \cong loc(Y).
\end{equation}

We say a local system $\cE\in loc(X\setminus K)$ is a \textit{KCH local system} if it comes from a KCH representation through the correspondence (\ref{rep-loc}).

\subsection{Unipotent KCH representation}\label{Sec:UniKCH}
The KCH representation requires not only that the action of the chosen meridian has an invariant subspace of codimension one, but also that the matrix of the action is diagonalizable. If we remove the diagonalizable condition, some more representations will be included. We study these representations in this section.

We first understand the action of the meridian. Given a fixed dimension $n$, let $I_n$ be the identity matrix,  and let $E_{ij}$ be the square matrix which is $1$ at entry $(i,j)$, and $0$ at all the other entries.

\begin{lem}\label{StdMatrix}
Let $V$ be a vector space with dimension $n\geq 1$ and $A\in GL(V)$. Suppose there is a subspace $W\subset V$ of codimension $1$ such that $A|_W = \mathrm{id}_W$. Then after choosing some basis, either (1) $A = I_n + cE_{11},  \text{ for some } c\neq 0, -1$; or (2) $n\geq 2$ and $A = I_n + E_{12}$.

\end{lem}
\begin{proof}
It is obvious when $n=1$. When $n \geq 2$, there exists a basis $\{v_1,\dotsb, v_n\}$ such that
$$
A= \begin{pmatrix}
c_1 & c_2 & \dotsb & c_n \\
0 & 1 & \dotsb & 0 \\
\vdots & \vdots & \ddots & \vdots \\
0 & 0 & \dotsb & 1 
\end{pmatrix}.
$$
Since $A$ is invertible, $c_1\neq 0$. If $c_1 \neq 1$, we can choose other basis elements $v_i' = v_i + (1-c_1)^{-1}c_iv_1$ for $i =2,\dots,n$, so that $A = \textrm{diag} \{c_1,1,\dotsb,1\} = I_n + (c_1-1)E_{11}$. If $c_1 =1$, we can change the basis so that there is at most one non-zero number among $c_1,\dotsb, c_n$. Without loss of generality, we assume that $c_2\neq 0$. Then $A = I_n+E_{12}$ under the new basis $v_1' = c_2v_1$, $v_2' = v_2$, and $v_i' = c_2v_i - c_iv_2$ for $i= 3,\dots,n$.

\end{proof}

Following the lemma, there are two possibilities if we only require the meridian action is trivial on a codimension one subspace. One of the cases is the KCH representation. We define the other case to be the unipotent KCH representation, which is termed so because the meridian matrix can be normalized to a unipotent matrix.

\begin{defn}
Suppose $V$ is a vector space of dimension $n\geq 2$, and $m\in \pi_K$ is a fixed meridian. A representation $\rho: \pi_K\rightarrow GL(V)$ is a \textit{unipotent KCH representation} if $\rho(m)$ is similar to $I_n + E_{12}$ by conjugation.

We say a local system $\cE_u\in loc(X\setminus K)$ is a \textit{unipotent KCH local system} if it comes from a unipotent KCH representation through the correspondence (\ref{rep-loc}).
\end{defn} 

Let $\{m_i\}_{i\in I}$ be the set of meridians generating the knot group. Let $(\rho, V)$ be a $n$ dimensional unipotent KCH representation. For each meridian $m_i$, we define a subspace
$$V_i = \textrm{im}\, (\textrm{id}_V - \rho(m_i)) \subset V,$$
which is always $1$ dimensional by definition. Define
$$V_0 = \textrm{Span}_k\{V_I\}.$$

\begin{lem}\label{UniKCHirrsubrep}
$V_0$ is an irreducible sub-representation of $V$.
\end{lem}
\begin{proof}
To show that $V_0$ is closed under the knot group action, it suffices to prove for meridian generators. For any $m_i$ and any $v_0\in V_0$,
$$\rho(m_i) (v_0)= - (\textrm{id}_V - \rho(m_i)) (v_0)+ \textrm{id}_V(v_0) \subset V_i + V_0 = V_0.$$
Therefore $V_0$ is closed under the action of any $m_i$, and further the  entire knot group, proving that $V_0$ is a sub-representation. The irreducibility follows from the fact that $\textrm{Span}_{k[\pi_K]}\{V_I\} = \textrm{Span}_{k[\pi_K]}\{v_i\}$ for any $i\in I$ and any non-zero $v_i\in V_i$, similar to the case of KCH representations.
\end{proof}

\begin{lem}\label{quottrivial2}
The quotient representation $V/V_0$ is trivial.
\end{lem}
\begin{proof}
The proof is similar to that of Lemma \ref{quottrivial1}.
\end{proof}

In the remaining of the subsection, we study some properties of the unipotent KCH representation. If $V_0$ has dimension $1$, then the unipotent KCH representation is an extension of trivial representations. If $V_0$ has dimension greater or equal to $2$, we show by an example that there exist irreducible unipotent KCH representations.

\begin{prop}
Let $\rho: \pi_K \rightarrow GL(V)$ be a unipotent KCH representation. If $V_0\subset V$ has dimension $1$, then $V$ is an extension of trivial representations.
\end{prop}
\begin{proof}
We always have the short exact sequence of representations (or $k[\pi_K]$-modules)
$$0\rightarrow V_0 \rightarrow V\rightarrow V/V_0 \rightarrow 0.$$
Since $V_0$ has dimension $1$, $V_i=V_0$ for all $i$. For each $i\in I$, the restriction of $\rho(m_i)$ to $V_i$ is identity. Therefore $V_0$ is a trivial representation. It follows from Lemma \ref{quottrivial2} that $V/V_0$ is also trivial.
\end{proof}

\begin{eg} The Wirtinger presentation of the knot group of the trefoil is
$$\pi_K = \la m_1,m_2,m_3 \ra / (m_3m_2 = m_2m_1 = m_1m_3).$$
More specifically, we consider a planar diagram of the trefoil knot with three strands and three crossings. Each strand gives rise to a meridian generator. Each crossing imposes a relation among the generators. There is a redundant relation.

We define a unipotent KCH representation $\rho: \pi_K\rightarrow GL(2,k)$  by
$$\rho(m_1) =
\begin{pmatrix}
1 & 1 \\
0 & 1
\end{pmatrix},
\quad
\rho(m_2) =
\begin{pmatrix}
1 & 0 \\
-1 & 1
\end{pmatrix},
\quad 
\rho(m_3) =
\begin{pmatrix}
2 & 1 \\
-1 & 0
\end{pmatrix}.
$$
It is straightforward to verify that the relations in the knot group are satisfied. We will argue that $\rho$ is irreducible. Observe that $(1,0)^t$ spans the invariant subspace of $\rho(m_1)$ and $(0,1)^t$ spans the invariant subspace of $\rho(m_2)$. Since they are transverse, there is no proper invariant subspace of the $\pi_K$-action, proving the irreducibility.
\end{eg}

\section{Sheaves}\label{SectionSheaves}

Suppose $Y$ is a smooth manifold. Let $Mod(Y)$ be the abelian category of sheaves of $k$-modules on $Y$, and $Sh(Y)$ the bounded dg derived category. The abelian category $Mod(Y)$ is equivalent to the subcategory of $Sh(Y)$ consisting of objects concentrated in degree zero. 

Let $Loc(Y)\subset Sh(Y)$ be the subcategory of locally constant sheaves. Then $loc(Y) = Loc(Y)\cap Mod(Y)$ is the category of local systems in the usual sense, namely $\pi_1(Y)$ representations.

\subsection{Singular support}\label{Sec:Ssupp}

To each sheaf $\cF\in Sh(Y)$, one can associate a closed conic subset $SS(\cF)\subset T^*Y$, called the \textit{singular support} or the \textit{micro-support} \cite[Definition 5.1.1]{KS}. Typical examples include: (1) a sheaf is locally constant if and only if its micro-support is contained in the zero section, (2) the constant sheaf supported on a closed submanifold $Z$ has its micro-support being $T^*_Z Y$.

Let $0_Y$ be the zero section of $T^*Y$ and $\dot{T}^*Y = T^*Y\setminus 0_Y$. Suppose $\Lambda \subset \dot{T}^*Y$ is a connected conic closed Lagrangian. By a theorem of Guillermou-Kashiwara-Schapira \cite{GKS}, the subcategory
$$Sh_{\Lambda}(Y) = \{\cF\in Sh(Y)\,|\, SS(\cF)\cap \dot{T}^*Y \subset \Lambda\},$$
is invariant under a homogeneous Hamiltonian isotopy. Let $Mod_\Lambda(Y) = Sh_\Lambda(Y)\cap Mod(Y)$.

\smallskip
The singular support has functorial behaviors. Using the property of the singular support in \cite{KS}, we give a characterization of the sheaves microsupported along the conormal bundle of a closed submanifold. It serves as a substitute definition of the singular support in our geometric setting.

Let $Y$ be a manifold, $i: Z\rightarrow Y$ a closed embedding and $j: Y\setminus Z \rightarrow Y$ the open embedding. Note that $i_*$ and $j_!$ are exact functors. If $S\subset T^*Y$ is a subset, we write $\dot{S} := S\cap \dot{T}^*Y$.

\begin{lem}\label{locloc}  If $\cF\in Sh(Y)$, then $\dot{SS}(\cF) \subset \dot{T}^*_Z Y$ 
if and only if $j^{-1}\cF \in Loc(Y\setminus Z)$ and $i^{-1}\cF \in Loc(Z)$. If $\cF\in Mod(Y)$, then $\dot{SS}(\cF) \subset \dot{T}^*_Z Y$ if and only if $j^{-1}\cF \in loc(Y\setminus Z)$ and $i^{-1}\cF \in loc(Z)$.
\end{lem}

\begin{proof} The second assertion follows from the first assertion by considering sheaves concentrated at degree $0$. We prove the first assertion in two directions.

(1) Suppose $\dot{SS}(\cF) \subset \dot{T}^*_Z Y$. We apply \cite[Proposition 5.4.5]{KS}. Since $j$ is an open embedding, we have $SS(j^{-1}\cF) = SS(\cF) \cap T^*(Y\setminus Z)\subset 0_{Y\setminus Z}$. Hence $j^{-1}\cF\in Loc(Y\setminus Z)$.

Similarly apply \cite[Proposition 5.4.5]{KS} to $i$, $SS(i^{-1}\cF) = i_d(i_\pi^{-1}SS(\cF))$ where $T^*Z \xleftarrow{i_d} T^*Y|_Z \xrightarrow{i_\pi} T^*Y$. Because of the short exact sequence of bundle morphisms $0\rightarrow T_Z^*Y\rightarrow T^*Y|_Z\xrightarrow{i_d} T^*Z\rightarrow 0 $, we deduce that $T_Z^*Y$ is in the kernel of $i_d$. Also observe that $i_{\pi}^{-1}$ is a restriction. Hence $SS(i^{-1}\cF)\subset 0_Z$, which gives $i^{-1}\cF\in Loc(Z)$.

(2) Apply the triangle inequality of singular support to $j_!j^{-1}\cF \rightarrow \cF \rightarrow i_*i^{-1}\cF \xrightarrow{+1}$, we have $SS(\cF) \subset SS(j_!j^{-1}\cF)\cup SS(i_*i^{-1}\cF)$. Because $i_*i^{-1}\cF$ is a locally constant sheaf on the submanifold $Z$, its singular support is contained in the conormal bundle $T^*_Z Y$.

It suffices to show that $\dot{SS}(j_!j^{-1}\cF)\subset \dot{T}^*_Z Y$. Since the singular support is locally defined, we can assume $Y = \bbR^n$ with coordinates $(y_1,\dotsb, y_n)$, and $Z= \{y_1=\dotsb =y_k =0\}$. Then $U\cong (\bbR^{k}\setminus \{0\})\times \bbR^{n-k}$. Let $p:(\bbR^{k}\setminus \{0\})\times \bbR^{n-k}\rightarrow (\bbR^{k}\setminus \{0\})$ be the projection. Because $j^{-1}\cF$ is locally constant, the restriction to each fiber of $p$ is also locally constant. By \cite[Proposition 5.4.5]{KS}, there is $\cH \in Loc(\bbR^{k}\setminus \{0\})$ such that $j^{-1}\cF = p^{-1}\cH$. Let $\tilde{j} : \bbR^k\setminus\{0\} \rightarrow\bbR^k$ be the open embedding and let $\tilde{p} : \bbR^k\times \bbR^{n-k} \rightarrow\bbR^k$ be the projection, we have
$$j_!j^{-1}\cF = j_!p^{-1}\cH = \tilde{p}^{-1}\tilde{j}_!\cH.$$
Observe that $\dot{SS}(\tilde{j}_!\cH)\subset \dot{T}_{0}^*\bbR^k$, then $\dot{SS}(j_!j^{-1}\cF) = \dot{SS}(\tilde{p}^{-1}\tilde{j}_!\cH) \subset \dot{T}^*_Z Y$. We complete the proof.

\end{proof}

We learn from the previous lemma that a sheaf microsupported along $T^*_Z Y$ is determined by a
local system on $Z$ and a local system on $Y \setminus Z$. The reversed direction is characterized by the study of $\textrm{Ext}^1_Y(i_*\cG,j_!\cH)$, where $\cH\in Loc(Y\setminus Z)$ and $\cG\in Loc(Z)$. In particular, if both $\cH$ and $\cG$ are concentrated at degree $0$, an extension class is presented by a short exact sequence of sheaves:
\begin{align}\label{extension}
0\rightarrow j_!\cH \rightarrow \cF \rightarrow i_*\cG \rightarrow 0.
\end{align}
The extension classes classify the possible gluings between the local systems. In fact, they only depend on $\cG$, and $\cH$ restricted to a neighborhood of $Z$. More precisely, we have

\begin{lem}\label{ext}
$\textrm{Ext}^1_{Y}(i_*\cG,j_!\cH) = R^0Hom_{Z}(\cG,i^{-1}Rj_*\cH).$
\end{lem}
\begin{proof}
Apply $RHom_{Y}(i_*\cG, -)$ to the distinguished triangle
$$j_!\cH\rightarrow Rj_*\cH\rightarrow Rj_*\cH|_{Z}\xrightarrow{+1}.$$
Because $j^{-1}\circ  i_* =0$, the middle term is $RHom_{Y}(i_*\cG,Rj_*\cH) = RHom_Z(j^{-1}i_*\cG, \cH) = 0$. The triangle implies $RHom_{Y}(i_*\cG,j_!\cH) = RHom_{Y}(i_*\cG,Rj_*\cH|_Z[-1])$. Continuing the calculation, we have
\begin{align*}
R\textrm{Hom}_{Y}(i_*\cG,j_!\cH) &= R\textrm{Hom}_{Y}(i_*\cG,Rj_*\cH|_Z[-1]), &&[\text{Explained}]\\
&= R\textrm{Hom}_{Y}(i_*\cG,Rj_*\cH|_Z)[-1], &&[\text{Degree shift}]\\
&= R\textrm{Hom}_{Y}(i_*\cG,i_*i^{-1}Rj_*\cH)[-1], &&[\text{Definition of the restriction}]\\
&= R\textrm{Hom}_{Z}(i^{-1}i_*\cG,i^{-1}Rj_*\cH)[-1], &&[\text{Adjunction}]\\
&= R\textrm{Hom}_{Z}(\cG,i^{-1}Rj_*\cH)[-1]. &&[i^{-1}i_* = \textrm{id}]
\end{align*}
Taking the cohomology at degree $1$ completes the proof.
\end{proof}

\subsection{Simple sheaves}\label{Sec:Simple}

Recall that $X = S^3$ or $\bbR^3$, and $K\subset X$ is an oriented knot. 

We denote the conormal bundle removing the zero section by $\Lambda_K = N^*_KX\cap \dot{T}^*X$.

Instead of the general definition of a simple sheaf in \cite[Definition 7.5.4]{KS}, we introduce a version in the context of knot conormals, beginning with a lemma.

\begin{lem}\label{SheafRep}

Let $\ell$ be the longitude and $m$ the preferred meridian of $K$. A sheaf $\cF \in Mod_{\Lambda_K}(X)$ is equivalent to the following data:

\begin{enumerate}
\item
a representation $\rho :\pi_K\rightarrow GL(V)$, and
\item
a representation $\rho': \bbZ_K\rightarrow GL(W)$, where the subscript represents the generator, and 
\item
a linear transform $T: W\rightarrow V$, such that (a) $\rho(\ell) \circ T= T\circ \rho'(K)$ and (b) $m$ acts on the image of $T$ as identity.
\end{enumerate}
\end{lem}

\begin{proof}
(1) Suppose $\cF\in Mod_{\Lambda_K}(X)$. By Lemma \ref{locloc}, we have $j^{-1}\cF\in loc(X\setminus K)$ and $i^{-1}\cF\in loc(K)$. The local systems give rise to the representations $\rho :\pi_K\rightarrow GL(V)$ and $\rho': \bbZ_K\rightarrow GL(W)$  (\ref{rep-loc}). By construction, $V$, $W$ are the stalks of $\cF$ on $X\setminus K$ and $K$. Because $K$ is a closed submanifold, the sheaf data give a restriction map $T: W\rightarrow V$.  

The restriction map has to be compatible with the $\pi_K$-action on $V$ and the $\bbZ_K$-action on $W$. Since the restriction map is local, we only need on $V$ the action of the subgroup $\bbZ_m\times \bbZ_\ell\subset \pi_K$. The compatibility is expressed as the conditions (3a) and (3b) for $T$.

(2) Conversely, assuming the list of data, we construct a sheaf in $Mod_{\Lambda_K}(\cF)$. The two representations determine $\cH\in loc(X\setminus K)$ and $\cG\in loc(K)$. The desired sheaf is determined by a class in $\textrm{Ext}^1_X(i_*\cG,j_!\cH)$, or $R^0\textrm{Hom}_{K}(\cG, i^{-1}Rj_*\cH)$ by Lemma \ref{ext}, where $\cH$ and $\cG$ are considered as complexes concentrated at degree $0$. Because $\cG$ is concentrated at degree $0$, classes in $R^0\textrm{Hom}$ are just closed maps. The sheaf $i^{-1}Rj_*\cH$ is described by the complex
$$0\longrightarrow V\xrightarrow{1-\rho(m)} V\longrightarrow 0,$$
together with an action $\rho(\ell)$ on $V$, which commutes with the differential of the complex. Now $T$ (with condition (3a)) determines a degree $0$ map $f$:
\begin{center}
\begin{tikzpicture}
  \node (A){$0$};
  \node (B)[right of=A, node distance=1.5cm]{$W$};
  \node (C)[right of=B, node distance=3cm]{$0$};
  \node (D)[right of=C, node distance=1.5cm]{$0$};
  \node (E)[below of=A, node distance=1.5cm]{$0$};
  \node (F)[below of=B, node distance=1.5cm]{$V$};
  \node (G)[below of=C, node distance=1.5cm]{$V$};
  \node (H)[right of=G, node distance=1.5cm]{$0$};
  \draw[->] (A) to node [] {$$} (B);
  \draw[->] (B) to node [swap]{$$}(C);
  \draw[->] (C) to node [swap]{$$}(D);
  \draw[->] (E) to node []{$$}(F);
  \draw[->] (F) to node []{$1-\rho(m)$}(G);
  \draw[->] (G) to node []{$$}(H);
    \draw[->] (C) to node []{$0$}(G);
    \draw[->] (B) to node [swap]{$T$}(F);
\end{tikzpicture}
\end{center}
Recall that a morphism $f: A^\bullet \rightarrow B^\bullet$ in the dg derived category has the differential $df = fd_A - (-1)^{\deg{(f)}}d_B f$. In our case, that is $df = -(1-\rho(m))\circ T$. It is zero because the condition (3b) that $\rho(m)$ acts on the image of $T$ as identity. Therefore $f$ is closed and we obtain the desired morphism, and further the desired sheaf.
\end{proof}

\begin{defn}\label{simplesheaf}
Suppose $T:W\rightarrow V$ is the linear transform determined by a sheaf $\cF\in Mod_{\Lambda_K}(X)$ as in Lemma \ref{SheafRep}. We say $\cF$ is \textit{simple} if $cone(T)$ has rank $1$. In other words, either
\begin{enumerate}
\item
$T$ is injective with a rank $1$ cokernel, or
\item
$T$ is surjective with a rank $1$ kernel.
\end{enumerate}
Let $Mod^s_{\Lambda_K}(X)\subset Mod_{\Lambda_K}(X)$ be the (no longer abelian) subcategory of simple sheaves.
\end{defn}

\subsection{Classification}\label{Sec:Classification}

Recall that $Mod_{\Lambda_K}^s(X)\subset Mod_{\Lambda_K}(X)$ is the subcategory of simple abelian sheaves microsupported along the knot conormal. In this section, we classify objects in $Mod^s_{\Lambda_K}(X)$.

We first show how to construct a simple sheaf from a KCH local system.

\begin{lem}\label{pushforwardKCH}
If $\cE \in loc(X\setminus K)$, then $j_*\cE \in Mod_{\Lambda_K}(X)$. If in addition $\cE$ is a KCH local system, then $j_*\cE \in Mod^s_{\Lambda_K}(X)$.
\end{lem}

\begin{proof}
(1) We apply Lemma \ref{locloc} to verify the first assertion. First $j^{-1}j_*\cE = \cE \in loc(X\setminus K)$. To verify $i^{-1}j_*\cE \in loc(K)$, because the singular support is locally defined, we assume $X = \bbR^3$ and $K= \{x_1 =x_2 =0\}$. Then $X\setminus K = (\bbR^{2}\setminus \{0\})\times \bbR$. Let $p:(\bbR^{2}\setminus \{0\})\times \bbR\rightarrow (\bbR^{2}\setminus \{0\})$ be the projection. There is $\cG \in Mod(\bbR^{2}\setminus\{0\})$ such that $\cE = p^{-1}\cG$. Let $\tilde{j} : \bbR^2\setminus\{0\} \rightarrow\bbR^2$ be the open embedding, then $SS(j_*\cG) \subset T^*_{0}\bbR^2$. Let $\tilde{p} : \bbR^2\times \bbR \rightarrow\bbR^2$ be the projection, then $SS(\tilde{p}^{-1}j_*\cG) \subset T^*_KX$. Since $p$ is a topological submersion of fiber dimension $1$, we have $p^{!} = p^{-1}[1]$, and further $j_*p^{-1}\cG = j_*p^!\cG[-1] =  \tilde{p}^!\tilde{j}_*\cG[-1] = \tilde{p}^{-1}\tilde{j}_*\cG$. Therefore, $SS(i^{-1}j_*\cE)=SS(i^{-1}j_*p^{-1}\cG)=SS(i^{-1}\tilde{p}^{-1}\tilde{j}_*\cG) \subset 0_K$. Therefore we have the desired $i^{-1}j_*\cE \in loc(K)$.

(2)
Now suppose $\cE$ is a KCH local system. Because the simpleness is a local property, we adopt the local chart as above and use the same notation. Given the induced $\cG\in loc(\bbR^2\setminus\{0\})$,
$$(\tilde{j}_*\cG)_{0} = \varinjlim_{0 \in U}\Gamma(U, \tilde{j}_*\cG) = \varinjlim_{0 \in U}\Gamma(\tilde{j}^{-1}(U), \cG) = \varinjlim_{0 \in U}\Gamma(U\setminus\{0\}, \cG).$$
If $U$ is an open ball containing $0$, then $U\setminus\{0\}$ is not simply connected and its fundamental group is generated by a meridian. Let $(\rho, V)$ be the representation determined by Lemma \ref{SheafRep}. Sections over $U\setminus \{0\}$ correspond to the vectors in $V$ that are invariant under the action of the meridian. Passing to the direct limit, the stalk (or $W$ as in Lemma \ref{SheafRep}) consists of such vectors as well.

In this case, the linear transform $T: W\rightarrow V$ is the natural inclusion. The previous arguments yield that $T$ is injective with a rank $1$ cokernel. Hence the sheaf $j_*\cE$ is simple. 
\end{proof}

\begin{defn}
Define $\cS$ to be the set of isomorphism classes of objects in $Mod^s_{\Lambda_K}(X)$. The isomorphism is given by the sheaf isomorphism.
\end{defn}

\begin{thm}\label{Sheafclassify}

A sheaf $\cF\in \cS$ is isomorphic to exactly one of the following:
\begin{enumerate}
\item 
$\cL_X \oplus j_! k_{X\setminus K}$;  
\item
$j_*\cE$, for a KCH local system $\cE\in loc(X\setminus K)$;
\item
$j_*\cE_u$, for a unipotent KCH local system $\cE_u\in loc(X\setminus K)$;
\item
$\cL_X \oplus i_*\cG_\alpha$, for a rank $1$ local system $\cG_\alpha\in loc(K)$, where $\alpha\neq 0$ is the monodromy;
\item
$\cL_X \oplus \cF'$, where $\cF'$ admits the non-splitting short exact sequence
$$0\rightarrow i_*k_K \rightarrow \cF' \rightarrow k_X\rightarrow 0.$$
\end{enumerate}
Here $\cL_X$ denotes a local system on $X$.
\end{thm}

\begin{proof}
Suppose $\cF\in Mod^s_{\Lambda_K}(X)$. Let $\rho: \pi_K\rightarrow GL(V)$, $\rho': \bbZ_K\rightarrow GL(W)$ and $T: W\rightarrow V$ be the data defined in Lemma \ref{SheafRep}. Since $\cF$ is simple, $cone(T)$ has rank $1$. We study each of the cases in Definition \ref{simplesheaf}.

(A) Suppose $0\rightarrow W\xrightarrow{T} V$ and $T$ has a rank $1$ cokernel. By Lemma \ref{SheafRep} (3b), the meridian $m$ acts as identity on the image of $T$. Therefore the codimension of the subspace of $V$ on which $m$ acts as the identity is either $0$ or $1$.

(A1) If the codimension is $0$, then the meridian $m$ acts trivially on $V$. Because any two meridians are conjugate, any meridian acts trivially. Because the knot group is generated by the meridians, the entire group acts trivially on $V$. In particular, the action of the longitude is trivial. By Lemma \ref{SheafRep} (3a), the action on $W$ is also trivial. Choose a splitting $V = \textrm{im}\,T \oplus V_0$. The sub-representation of $\textrm{im}\,T$ together with $W$ determines a constant sheaf $\cL\in loc(X)$. The subreprensentation $V_0$ determines $k_{X\setminus K}\in loc(X\setminus K)$, giving $\cF' = j_!k_{X\setminus K}$.

(A2) If the codimension is $1$, then $\cE :=j^{-1}\cF$ is either a KCH local system or a unipotent KCH local system by Lemma \ref{StdMatrix}. Because $j^{-1}$ is left adjoint to $j_*$, we have a natural morphism 
$$\cF\rightarrow j_*j^{-1}\cF= j_*\cE.$$
We prove the morphism is an isomorphism by checking the stalk at each point. If $x\in X\setminus K$, we have $(j_*j^{-1}\cF)_x = (j^{-1}\cF)_x = \cF_x$ because $j$ is an open embedding. If $x\in K$, and suppose $m$ is a meridian which bounds a disk intersecting $K$ transversely at $x$, then $\cF_x = W$ by Lemma \ref{SheafRep}, and 
\begin{align}\label{stalkatapointonknot}
\begin{split}
(j_*j^{-1}\cF)_x &= \varinjlim_{x\in U}\Gamma(U,j_*j^{-1}\cF) \\
		        &= \varinjlim_{x\in U}\Gamma(j^{-1}(U), j^{-1}\cF) \\
		        &= \varinjlim_{x\in U}\Gamma(U\cap(X\setminus K), \cF) = V^{\la m \ra},
\end{split}
\end{align}
where $V^{\la m \ra}$ is the subspace of $V$ on which $m$ acts as identity.

Being a KCH or unipotent KCH representation implies that $V^{\la m \ra}\subset V$ has codimension $1$. We identify $W$ as a subspace of $V$ through the map $T$. Both $W$ and $V^{\la m \ra}$ are codimensional $1$ subspaces of $V$ on which $m$ acts as the identity. Therefore $V^{\la m \ra} = W$. We have checked that the stalk of $\cF\rightarrow j_*\cE$ at each point is an isomorphism. Hence $\cF = j_*\cE$. 

If $\cE$ is a KCH representation, we get case (2) of the theorem.

If $\cE$ is a unipotent KCH representation, we get case (3) of the theorem.

(B) Suppose $W\xrightarrow{T} V\rightarrow 0$ and $T$ has a rank $1$ kernel. Since $V = \textrm{im}\,T$, it is invariant under the $m$-action, and further invariant under the $\pi_K$-action. Suppose $\{w_0,w_1,\dotsb, w_n\}$ is a basis of $W$ such that $w_0$ spans $\ker T$ (which is unique up to scalar multiplication). Note that $W' := \textrm{Span}_k\{w_1,\dotsb,w_n\}$ satisfies $T(W') = V$. Let $A = \rho'(K)$, which is an invertible matrix acting on $W$. We study the eigenspaces of $A$. Since $\ker T$ is a sub-representation, $Aw_0 = c_0w_0$ for some $c_0\in k^*$. We also have $T\circ A(w_i) = \rho(\ell) \circ (T(w_i)) = T(w_i)$, first by the property of $T$ and second because the $\pi_K$-action is trivial. Because $w_0$ spans $\ker T$ and $T$ identifies $W'$ with $V$, we have $Aw_i = w_i + c_iw_0$ for $i=1,\dots, n$. Now
$$
A= \begin{pmatrix}
c_0 & c_1 & \dotsb & c_n \\
0 & 1 & \dotsb & 0 \\
\vdots & \vdots & \ddots & \vdots \\
0 & 0 & \dotsb & 1 
\end{pmatrix}.
$$

(B1) If $c_0 \neq 1$, we can choose other basis elements $w_i' = w_i + (1-c_0)^{-1}c_iw_0$ for $i =1,\dots,n$, such that $A = \textrm{diag} \{c_0,1,\dotsb,1\}$. Then $\cF \cong \cL_X\oplus i_*\cG_\alpha$. The local system $\cL_X$ is determined by $W'$, $V$ and $T$. The rank $1$ local system $\cG_\alpha\in loc(K)$ with $\alpha\neq 1$ is determined by $\ker T$.

(B2) If $c_0=1$, and all the other $c_i=0$, then $\cF \cong \cL_X \oplus i_*\cG_1$, where $\cG_1 = k_K$ is the constant sheaf on the knot. Together with (B1), we obtain case (4) of the theorem.

(B3)
If $c_0 =1$, and some of the other $c_i\neq 0$. By Lemma \ref{StdMatrix}, there is a basis such that $A = I_{n+1} + E_{12}$. We also compute that $\textrm{Ext}^1_X(k_X, i_*k_K) = k$. Then $\cF  = \cL_X \oplus \cF'$, where $\cL_X\in loc(X)$ is a rank $n-1$ local system, and $\cF'$ admits the short exact sequence $0\rightarrow i_*k_K \rightarrow \cF' \rightarrow k_X\rightarrow 0$. This is the last case in the assertion.

We complete the proof.
\end{proof}

\subsection{Moduli}\label{Sec:Moduli}
We defined $\cS$ to be the set of isomorphism classes of objects in $Mod^s_{\Lambda_K}(X)$, namely the simple abelian sheaves microsupported along the knot conormal. In this section we define a quotient set of ``simple sheaves up to local systems''. 

\begin{defn}

Let $\cM$ be the quotient set of $\cS$, by the equivalence relation generated by the following relations.

If for some local system $\cL_X \in loc(X)$ and two sheaves $\cF_1,\cF_2\in Mod^s_{\Lambda_K}(X)$, there is a short exact sequence of $\cF_1,\cF_2$, and $\cL$, namely (ER1) $0\rightarrow \cL_X \rightarrow \cF_1 \rightarrow \cF_2\rightarrow 0$; or (ER2) $0\rightarrow \cF_1\rightarrow \cL_X \rightarrow \cF_2\rightarrow 0$; or (ER3) $0\rightarrow \cF_1 \rightarrow \cF_2\rightarrow \cL_X \rightarrow 0$, then $\cF_1$ and $\cF_2$ are equivalent, denoted by $\cF_1\sim \cF_2$. We also force it to be a symmetric relation, i.e. if $\cF_1 \sim \cF_2$, then $\cF_2\sim \cF_1$ as well.

\end{defn}
\begin{rmk}
Two elements $\cF,\cF' \in \cS$ are isomorphic in $\cM$, if there exists a sequence of elements $\cF_1,\dotsb, \cF_n \in \cS$ such that 
$$\cF = \cF_1 \sim \cF_2 \sim \dotsb \sim \cF_n = \cF'.$$ 
\end{rmk}

\begin{rmk} The equivalence relations defining $\cM$ make better sense in the derived category, where short exact sequences are replaced by distinguished triangles. Especially, (ER2) becomes
$$ \cL_X \rightarrow \cF_2\rightarrow \cF_1[1] \xrightarrow{+1},$$
which looks more similar to the other two equivalence relations. For now, the ad hoc definition works for our purpose of classifying simple sheaves up to local systems in the abelian category.
\end{rmk}

\begin{prop}\label{Sheafuptoloc}
An isomorphism class in $\cM$ is represented by exactly one of the following:
\begin{enumerate}
\item
$j_*\cE$, for an irreducible KCH local system $\cE\in loc(X\setminus K)$; or
\item
$j_*\cE_u$, for an irreducible unipotent KCH local system $\cE_u\in loc(X\setminus K)$; or
\item
$i_*\cG_\alpha$, for a rank $1$ local system $\cG_\alpha\in loc(K)$.\footnote{In a private conversation, Lenhard Ng explained that he used to construct some ``dimension-$0$ degenerate'' KCH representations, but the definition  was not written down. It probably corresponds to this case.}
\end{enumerate}
\end{prop}

\begin{proof}
Objects in $Mod^s_{\Lambda_K}(X)$ are explicitly written down in Theorem \ref{Sheafclassify}. (ER1) allows us to set $\cL_X=0$ whenever there is a direct summand $\cF = \cF'\oplus\cL_X$ because of the natural short exact sequence $0\rightarrow \cF' \rightarrow \cF \rightarrow \cL_X\rightarrow 0$.

For the first case of Theorem \ref{Sheafclassify}, observe that $j_!k_{X\setminus K}$ and $i_*k_K$ are in the same class, because of (ER2) and the short exact sequence $0 \rightarrow j_!k_{X\setminus K}\rightarrow k_X \rightarrow i_*k_K\rightarrow 0$. In the last case when $0\rightarrow k_K \rightarrow \cF \rightarrow k_X\rightarrow 0$, we have $\cF$ equivalent to $i_*k_K$  by definition. Therefore cases (1), (4), and (5) of Theorem \ref{Sheafclassify} all correspond to case (3) here.

Next we focus on case (2) and (3) of Theorem \ref{Sheafclassify}, when the sheaf comes from a KCH or a unipotent KCH local system. Suppose $\cE$ is a KCH local system corresponding to the KCH representation $(\rho, V)$. The meridian subspace, defined in (\ref{Merdiansubspace1}), gives rise to a sub-representation $(\rho, V_0)$ which is irreducible and also KCH. Let $\cE_0$ be the associated KCH local system. Since $\cE_0\subset \cE$ is a subsheaf, there is a short exact sequence of sheaves
$$0\rightarrow j_*\cE_0\rightarrow j_*\cE \rightarrow j_*\cE/j_*\cE_0 \rightarrow 0.$$
We will argue that $j_*\cE/j_*\cE_0$ is a local system on $X$, then by (ER3) any KCH representation is equivalent to an irreducible one. First apply the exact functor $j^{-1}$, we have $j^{-1}( j_*\cE/j_*\cE_0) = \cE/\cE_0$, which is a trivial local system on $X\setminus K$ by Lemma \ref{quottrivial1}. Second for any point $x\in K$, we restrict to a contractible open neighborhood $U$ containing $x$. Since $\cF\mapsto \cF|_U$ is also an exact functor, we have the short exact sequence 
\begin{equation}\label{KCHreprestrictedtoU}
0\rightarrow j_*\cE_0|_U\rightarrow j_*\cE|_U \rightarrow (j_*\cE/j_*\cE_0)|_U \rightarrow 0.
\end{equation}
Locally, we can assume $U = \bbR^3$, $K = \{x_1 = x_2 = 0\}$. There is a unique meridian $m$ up to homotopy. Taking the stalk at $x$ (which is also an exact functor) of (\ref{KCHreprestrictedtoU}), there is 
$$0\rightarrow (j_*\cE_0|_U)_x \rightarrow (j_*\cE|_U)_x \rightarrow ((j_*\cE/j_*\cE_0)|_U)_x \rightarrow 0.$$
Unwrapping the definition (in the way of (\ref{stalkatapointonknot})), we see $(j_*\cE_0|_U)_x = \Gamma(U, \cE_0) = V_0^{\la m \ra}$, and $(j_*\cE|_U)_x = \Gamma(U, \cE) = V^{\la m \ra}$. The superscript refers to the invariant subspace under the action of $m$. Hence $((j_*\cE/j_*\cE_0)|_U)_x = V^{\la m \ra}/V_0^{\la m \ra}$. By construction, it is isomorphic to $V/V_0$, which is the stalk of $\cE/\cE_0$ at any point on $X\setminus K$. 
Therefore any KCH local system is equivalent to an irreducible KCH local system in $\cM$. The proof for the unipotent KCH representation is similar (where Lemma \ref{quottrivial1} is replaced by Lemma \ref{quottrivial2}).

Now we have reduced to the three cases in the assertion. It remains to show that any two cases are not equivalent. Our basic strategy is to assume that there exists one of the short exact sequence in the equivalence relations, and then to derive a contradiction by restricting to $K$ or to $X\setminus K$.

We first consider $j_*\cE$ for an irreducible KCH representation. There cannot be (ER1) or (ER3), because we can restrict to $X\setminus K$ and the resulting short exact sequence yields that $\cE$ has a proper subsheaf, which contradicts to that $\cE$ is irreducible. Assume we have (ER2), i.e. a short exact sequence  $0\rightarrow \cF_1\rightarrow \cL_X \rightarrow \cF_2\rightarrow 0$. If $\cF_1 = j_*\cE$, by restricting to $K$, we first see that $\cF_2$ cannot be the push forward of a KCH or a unipotent KCH local system by dimension reasons. The only possibility left is that $\cF_2 = i_*\cG_\alpha$ for a local system $\cG_\alpha\in loc(K)$. Restricting to the knot complement $X\setminus K$, we see that $\cE$ is isomorphic to $\cL_X|_U$. However $X$ is simply connected, yielding $\cL_X = k_X^{\oplus n}$, and therefore $\cL_X|_U$ is a trivial local system, which cannot be isomorphic to $\cE$. A similar argument holds for $\cF_2 = j_*\cE$. We conclude that $j_*\cE$ for an irreducible KCH local system is not equivalent to any other cases in the assertion.

The argument for an irreducible unipotent KCH local system is similar.

Finally we consider $i_*\cG_\alpha$. Any two distinct $\alpha \neq \alpha'$ give non-isomorphic sheaves. It is straightforward to check they are not equivalent in the moduli set. It is neither equivalent to the push forward of an irreducible KCH nor an irreducible unipotent KCH local system by the previous argument.

We complete the proof.

\end{proof}

\subsection{Derived sheaves}\label{Sec:DerivedSheaves}
The notion of the simpleness is generally defined for an object in the derived category $Sh_{\Lambda_K}(X)$. Let $Sh^s_{\Lambda_K}(X) \subset Sh_{\Lambda_K}(X)$ be the subcategory of simple sheaves. We are able to describe the objects in this category up to finite extensions with locally constant sheaves. In fact, the moduli set of ``simple (derived) sheaves up to locally constant sheaves'', which we term as $\widehat{\cM}$, up to degree shifts is isomorphic to $\cM$.

We quickly explain the simpleness in our geometric setting. Similar to Lemma \ref{SheafRep}, a sheaf $\cF \in Sh_{\Lambda_K}(X)$ determines the following data:
\begin{enumerate}
\item
a chain complex of $k[\pi_K]$-module $V^\bullet$; and
\item
a chain complex of $k[\bbZ_K]$-module $W^\bullet$; and
\item
a chain map $T^\bullet: W^\bullet\rightarrow V^\bullet$ with compatibility conditions.
\end{enumerate}
The sheaf is simple if $cone(T^\bullet)\cong k[d]$ for some integer $d$.

\begin{rmk}
In the abelian case, a sheaf is equivalent to the list of data, while in the derived case, the sheaf contains more information. For example, the sheaf restricted to the knot complement is a locally constant sheaf, which depends on the simplicial set structure rather than just the knot group. See \cite{Tr} or \cite[Appendix A]{Lu}. 
\end{rmk}

Recall $Loc(X)\subset Sh_{\Lambda_K}(X)$ is the subcategory of locally constant sheaves. Consider the set of isomorphism classes of objects in the quotient $Sh_{\Lambda_K}(X)/Loc(X)$. Specifically, we write $\cF_1\sim \cF_2$ and $\cF_2\sim \cF_1$ for two objects $\cF_1,\cF_2\in Sh_{\Lambda_K}(X)$, if there exists a locally constant sheaf $\cL_X \in Loc(X)$ and a distinguished triangle,
$$\cF_1\rightarrow \cF_2\rightarrow \cL_X \xrightarrow{+1}.$$
Then two objects $\cF,\cF'\in Sh_{\Lambda_K}(X)$ are isomorphic in $Sh_{\Lambda_K}(X)/Loc(X)$ if there are intermediate objects $\cF_1,\dotsb, \cF_n\in Sh_{\Lambda_K}(X) $ such that
$$\cF = \cF_1 \sim \cF_2 \sim \dotsb \sim \cF_n = \cF'.$$

Since the simpleness passes to the quotient, we define $\widehat{\cM}$ to be the set of isomorphism classes of simple objects in $Sh_{\Lambda_K}(X)/Loc(X)$.

\begin{lem}\label{eqtohomology}
Suppose $\cF\in Sh_{\Lambda_K}(X)$ satisfies that $H^i\cF \in loc(X)$ for all $i\neq d$, then $\cF \cong H^d\cF[-d]$ in $Sh_{\Lambda_K}(X)/Loc(X)$.
\end{lem}
\begin{proof}
By hypothesis, $d$ is the degree where we may not have a local system. Let $\tau$ be the truncation functor. Consider the distinguished triangle
$$\tau_{< d}\,\cF \rightarrow \cF \rightarrow \tau_{\geq d}\,\cF \xrightarrow{+1}.$$
By construction, $H^i(\tau_{< d}\,\cF) = H^i(\cF)$ when $i<d$ and $H^i(\tau_{< d}\,\cF)=0$ when $i\geq d$. In either case, we have $H^i(\tau_{< d}\,\cF)\in loc(X)$. Therefore $\tau_{< d}\,\cF\in Loc(X)$, and $\cF \cong \tau_{\geq d}\,\cF$ in $Sh_{\Lambda_K}(X)/Loc(X)$.

Similarly, applying $\tau_{\leq d}\rightarrow \textrm{id}\rightarrow \tau_{>d}\xrightarrow{+1}$ to $\tau_{\geq d}\,\cF$, we show $\tau_{\leq d}\, \tau_{\geq d}\,\cF\cong \tau_{\geq d}\,\cF$ in the quotient. Since $\tau_{\leq d}\, \tau_{\geq d}\,\cF = H^d \cF[-d]$, the assertion follows.
\end{proof}

\begin{prop}\label{DerivedSheafloc} An isomorphism class in $\widehat\cM$ is represented by exactly one of the following:
\begin{enumerate}
\item
$j_*\cE[-d]$, for an irreducible KCH local system $\cE\in loc(X\setminus K)$, $d\in\bbZ$; or 
\item
$j_*\cE_u[-d]$, for an irreducible unipotent KCH local system $\cE_u\in loc(X\setminus K)$, $d\in\bbZ$; or
\item
$i_*\cG_\alpha[-d]$, for a rank $1$ local system $\cG_\alpha\in loc(K)$, $d\in\bbZ$.
\end{enumerate}
\end{prop}
\begin{proof}
From the earlier part of this subsection, a sheaf $\cF\in Sh^s_{\Lambda_K}(X)$ determines a chain map $T^\bullet: V^\bullet \rightarrow W^\bullet$. Taking cohomology, we get $H^iT^\bullet: H^iV^\bullet \rightarrow H^iW^\bullet$. Since $cone(T^\bullet)$ has rank $1$, there is precisely one degree such that $cone(H^iT^\bullet)$ is not zero with rank $1$.

Because taking the stalk is an exact functor, it commutes with the kernel and the cokernel, and hence commutes with the homology functor. Therefore $H^i\cF$ is equivalent to the data of a $k[\pi_K]$-module $H^iV^\bullet$, a $k[\bbZ_K]$-module $H^iW^\bullet$, and a linear transform $H^iT^\bullet: H^iV^\bullet \rightarrow H^iW^\bullet$ with compatibility conditions. What we get in the last paragraph can be rephrased as $H^i\cF$ is a simple sheaf concentrated at one degree, and zero otherwise. By Lemma \ref{eqtohomology}, $\cF \cong H^{d}\cF[-d]$ in $Sh_{\Lambda_K}(X)/Loc(X)$. 

Since $H^d\cF\in Mod^s_\Lambda(X)$ is a simple sheaf concentrated at degree zero, we can apply Proposition \ref{Sheafuptoloc} and obtain the list in the assertion. It remains to show that any two representatives are not equivalent. A morphism between $\cF_1,\cF_2 \in Sh_\Lambda(X)/Loc(X)$ is represented by a roof $\cF_1\leftarrow \cF\rightarrow \cF_2$. Suppose that we have a distinguished triangle
$$\cF\rightarrow \cF_2\rightarrow \cL_X\xrightarrow{+1},$$
where (1) $\cF\in Sh_{\Lambda_K}(X)$, (2) $\cF_2 \in Sh^s_{\Lambda_K}(X)$ is a representative in the assertation, and (3) $\cL_X\in Loc(X)$. Suppose $H^d\cF_2 \neq 0$. Taking cohomology, we get a long exact sequence in $Mod(X)$:
$$0\rightarrow \cL_{d-1} \rightarrow H^d\cF\xrightarrow{u} H^d\cF_2 \xrightarrow{v} \cL_{d}\rightarrow H^{d+1}\cF\rightarrow 0,$$
where $\cL_{d-1},\cL_{d} \in loc(X)$ are local systems on $X$. Because $H^d\cF_2$ is irreducible, we have either $\ker v$ is equal to either $0$ or $H^d\cF_2$. Similarly $\textrm{im}\,u$ is equal to either $0$ or $H^d\cF_2$.

If $H^{d}\cF$ is simple and $H^{d+1}\cF\in loc(X)$, we claim that $v=0$. Suppose otherwise, $v$ must be injective since $\ker v$ is equal to either $0$ or $H^d\cF_2$. In particular, there is a short exact sequence,
$$0\rightarrow H^d\cF_2 \xrightarrow{v} \cL_{d}\rightarrow H^{d+1}\cF\rightarrow 0.$$
Note that $\cL_{d}$ and $H^{d+1}\cF$ are local systems on $X$. The short exact sequence cannot hold, because any representative in the assertion does not fit into a two-term resolution by local systems on $X$. We can see a contradiction that the dimension will not match after we take the stalk at a point either on the knot or on the knot complement. Following the claim that $v=0$, we have a short exact sequence $0\rightarrow \cL_{d-1} \rightarrow H^d\cF\xrightarrow{u} H^d\cF_2\rightarrow 0$, and then the argument reduces to the abelian case in Proposition \ref{Sheafuptoloc}.

If $H^{d+1}\cF$ is simple and $H^d\cF\in loc(X)$, then $v\neq 0$ (because otherwise $\cL_d\cong H^{d+1}\cF$, a contradiction) and $u=0$ (because $\ker v =\textrm{im}\, u$). We have a short exact sequence $0\rightarrow H^d\cF_2 \xrightarrow{v} \cL_{d}\rightarrow H^{d+1}\cF\rightarrow 0$, again studied in Proposition \ref{Sheafuptoloc}.

In either case, we have verified that if $\cF\rightarrow \cF_2$ is an isomorphism in the quotient category, then the two sheaves must have the same representative. The same arguments hold for $\cF \rightarrow \cF_1$. Hence no pair of the sheaves in the list are equivalent. We complete the proof.
\end{proof}

\begin{cor}\label{moduliquotient}
Let $\bbZ$ acts on $\widehat{\cM}$ by degree shift, then $\widehat{\cM}/\bbZ \cong \cM$.
\end{cor}
\begin{proof}
It follows from Proposition \ref{Sheafuptoloc} and Proposition \ref{DerivedSheafloc}.
\end{proof}
Since the action is free, we can rewrite the relation as:
$$\widehat{\cM} = \cM \times \bbZ.$$

\section{Augmentations}\label{SectionAugs}

We introduce augmentations in this section. In the standard context, augmentations are defined based on the Legendrian contact dga, whose coefficient ring contains the full second relative homology class. However in the specialized version that we will discuss, it suffices to just introduce the notion of the (framed) cord algebra. The definition was first introduced in \cite{Ng3}. We adopt the convention in a later update \cite{Ng5}.

\subsection{Augmentations}\label{Sec:Augs}

The framing of an oriented knot is a choice of generators of $H_1(\del n(K))$. Suppose $\lambda,\mu$ are the classes of the longitude $\ell$ and a preferred meridian $m$, we can identify $\bbZ [H_1(\del n(K))]$ with $\bbZ[\lambda^{\pm1}, \mu^{\pm1}]$.

The \textit{framed cord algebra} $\cP_K$ of $K$ is the tensor algebra over $\bbZ[\lambda^{\pm1}, \mu^{\pm1}]$ freely generated by the elements in $\pi_K$, modulo the relations:
\begin{enumerate}
\item
$[e] = 1- \mu$;
\item
$[\gamma m] = [m\gamma] = \mu[\gamma]$ and $[\gamma \ell] = [\ell \gamma] = \lambda[\gamma]$, for $\gamma\in \pi_K$;
\item 
$[\gamma_1\gamma_2] - [\gamma_1 m\gamma_2] - [\gamma_1][\gamma_2] =0$, for $\gamma_1,\gamma_2 \in \pi_K$,
\end{enumerate}
where $e\in \pi_K$ is the identity element, and $[\gamma]$ is a generator of $\cP_K$ for any $\gamma\in \pi_K$.

\begin{defn}\label{AugDef}
An \textit{augmentation} is a unit-preserving algebra homomorphism $$\epsilon: \cP_K \rightarrow k.$$

Equivalently, one can define an augmentation by assigning to each generator in $\cP_K$ an element in $k$, preserving the relations above. Explicitly, it means:
\begin{enumerate}
\item (Normalization) $\epsilon([e])=1-\epsilon(\mu)$;
\item (Meridian) $\epsilon([m\gamma])=\epsilon([\gamma m]) =\epsilon(\mu[\gamma]) $;
\item (Longitude) $\epsilon([\ell\gamma])=\epsilon([\gamma \ell]) =\epsilon(\lambda[\gamma])$;
\item (Skein relations) $\epsilon([\gamma_1\gamma_2]) - \epsilon([\gamma_1 m \gamma_2])= \epsilon([\gamma_1])\epsilon([\gamma_2])$;
\end{enumerate}
for any $\gamma,\gamma_1,\gamma_2\in \pi_K$.
\end{defn}

\begin{rmk}\label{rmkAug}
Some of the relations in Definition \ref{AugDef} are redundant.
\begin{enumerate}
\item
The meridian relations are always satisfied, as long as the normalization and skein relations are satisfied. Taking $\gamma_1 = e$, $\gamma_2 = \gamma$ in the skein relation, we have
$$\epsilon([\gamma]) - \epsilon([m\gamma]) = \epsilon([e])\epsilon([\gamma]) = (1-\epsilon(\mu))\epsilon([\gamma]),$$
with the first equality by the skein relation and the second equality by the normalization. Organizing the terms we get $\epsilon([m\gamma]) = \epsilon(\mu[\gamma])$. Similarly we have $\epsilon([\gamma m ]) = \epsilon([\gamma] \mu) = \epsilon(\mu [\gamma])$.

\smallskip
\item The longitude relations are reduced in some cases.

If $\epsilon(\mu)\neq 1$ (or $\epsilon([e])\neq 0$ by normalization), then $\epsilon([\ell\gamma])=\epsilon([\gamma \ell])$ is automatically satisfied. Set $\gamma_1 = \ell$, $\gamma_2 = \gamma$ in the skein relation:
$$\epsilon([\ell \gamma]) - \epsilon([\ell m \gamma]) = \epsilon([\ell])\epsilon([\gamma]).$$
Because $\ell$ and $m$ commute, $\epsilon([\ell m \gamma]) = \epsilon([m \ell \gamma])$, which further equals $\epsilon(\mu [\ell\gamma] )$ by the meridian relation. Organizing the terms, we have
\begin{equation}\label{Long1}
(1 - \epsilon(\mu))\epsilon([\ell\gamma]) = \epsilon([\ell])\epsilon([\gamma]).
\end{equation}
Similarly we compute 
$(1 - \epsilon(\mu))\epsilon([\gamma\ell]) = \epsilon([\gamma])\epsilon([\ell])$. Since $1-\epsilon(\mu)\neq 0$, we have $\epsilon([\ell\gamma]) = (1 - \epsilon(\mu))^{-1}\epsilon([\ell])\epsilon([\gamma]) = \epsilon([\gamma\ell])$. The assertion is verified.

Furthermore, if $\epsilon(\mu) \neq 1$ and $\epsilon([\ell]) = \epsilon(\lambda)\epsilon([e])$, then we claim that
$$\epsilon([\ell\gamma]) =\epsilon(\lambda[\gamma]) \quad \textrm{for all}\; \gamma\neq e$$
as well. By the normalization and this hypothesis, (\ref{Long1}) implies $\epsilon([e])\epsilon([\ell\gamma]) = \epsilon([\ell])\epsilon([\gamma]) = \epsilon(\lambda)\epsilon([e])\epsilon([\gamma])$. Cancelling $\epsilon([e])$ (which is nonzero by assumption), we verify the desired assertion.
\end{enumerate}
\end{rmk}

\smallskip

We will define a map which sends a simple sheaf concentrated at degree zero to an augmentation.

Recall that $\cS$ is the set of isomorphism classes of simple abelian sheaves.  Also recall that a sheaf $\cF \in Mod_{\Lambda_K}(X)$ determines two representations:
\begin{enumerate}
\item
a representation $\rho :\pi_K\rightarrow GL(V)$, and
\item
a representation $\rho': \bbZ_K\rightarrow GL(W)$.
\end{enumerate}

\begin{defn}
Given a sheaf $\cF \in Mod_{\Lambda_K}(X)$, we define $\epsilon_\cF: \cP_K\rightarrow k$ by
\begin{align*}
\epsilon_\cF(\lambda) &= \textrm{tr}(\rho(\ell)) - \textrm{tr}(\rho'(K)); \\
\epsilon_\cF(\mu) &= \textrm{tr}(\rho(m)) -\dim V + 1; \\
\epsilon_\cF([\gamma]) &= \textrm{tr}\left(\rho(\gamma) -\rho(m\gamma) \right),
\end{align*}
where $\textrm{tr}$ stands for the trace of an operator.
\end{defn}

\begin{thm}\label{sheaftoaug}
If $\cF\in \cS$, then $\epsilon_\cF$ is an augmentation.
\end{thm}

\begin{proof} 
We will go through Definition \ref{AugDef}. To verify the normalization, we compute
$$\epsilon_\cF([e]) = \textrm{tr} (\rho(e) -\rho(m)) = \textrm{tr} (\textrm{id}_V -\rho(m)) = \dim V - \textrm{tr} (\rho(m)),$$
and
$$\epsilon_\cF(1-\mu) = 1 -\epsilon_\cF(\mu) = 1- (\textrm{tr}(\rho(m)) -\dim V +1) = \dim V - \textrm{tr} (\rho(m)).$$
The two sides equal.

By Remark \ref{rmkAug} (1), the meridian relations always hold if we can show that the skein relations are satisfied.

To see the longitude and skein relations, we apply Theorem \ref{Sheafclassify} and verify them case by case. It is straightforward to check that a direct summand with $\cL_X$ does not change the augmentation. In other words, if $\cF_1 = \cF_2 \oplus \cL_X$, then $\epsilon_{\cF_1} = \epsilon_{\cF_2}$. We assume $\cL_X = 0$ in all the cases of Theorem \ref{Sheafclassify}.

In either case (1) or (5) of Theorem \ref{Sheafclassify}, namely $\cF \cong j_!k_{X\setminus K}$ or $\cF$ admits the short exact sequence $0\rightarrow k_K\rightarrow \cF\rightarrow k_X\rightarrow0$, $(\rho, V)$ is a constant rank $1$ representation. We have $\epsilon_\cF([\gamma]) = 1- 1= 0$ for any $\gamma\in \pi_K$. The longitude relation and the skein relations are satisfied.

In case (4) of Theorem \ref{Sheafclassify}, $\cF = i_*\cG_\alpha$ for a rank $1$ local system $\cG_\alpha$ on $K$. By construction we have $V = 0$. Therefore $\epsilon_\cF([\gamma]) = 0$, and the longitude and skein relations follow.

Finally in case (2) and (3) of Theorem \ref{Sheafclassify}. We have $\cF = j_*\cE$ for a KCH or unipotent local system. By Lemma \ref{StdMatrix}, there is a basis of $V$ such that $M := \textrm{id}_V - \rho(m)$ equals to
\begin{itemize}
\item
$cE_{11}$ for some $c\neq 0,-1$, if $\cE$ is a KCH local system; or
\item
$E_{12}$, if $\cE$ is a unipotent KCH local system.
\end{itemize}
We check the skein relations. Let $A = \rho(\gamma_1)$ and $B= \rho(\gamma_2)$, then
\begin{align*}
\epsilon_\cF([\gamma_1\gamma_2]) - \epsilon_\cF([\gamma_1 m \gamma_2]) &= \textrm{tr}(\rho(\gamma_1\gamma_2) - \rho(m\gamma_1\gamma_2)) - \textrm{tr}( \rho(\gamma_1 m \gamma_2) - \rho(m\gamma_1 m \gamma_2)), \\
	&= \textrm{tr}\big((\textrm{id}_V- \rho(m))\rho(\gamma_1)(\textrm{id}_V- \rho(m))\rho(\gamma_2)\big) = \textrm{tr}(MAMB).
\end{align*}
and 
\begin{align*}
\epsilon_\cF([\gamma_1])\epsilon_\cF([\gamma_2]) = \textrm{tr}\big((\textrm{id}_V- \rho(m))\rho(\gamma_1)\big)\textrm{tr}\big((\textrm{id}_V- \rho(m))\rho(\gamma_2)\big) = \textrm{tr}(MA)\textrm{tr}(MB).
\end{align*}
If $\cE$ is a KCH local system, both equal to $c^2A_{11}B_{11}$. If $\cE$ is a unipotent KCH local system, both equal to $A_{21}B_{21}$. The skein relations are verified.

To see the longitude relation. If $M = cE_{11}$ with $c\neq 0, -1$, then $\epsilon_{\cF}(\mu) \neq 1$. By Remark \ref{rmkAug} (2), it suffices to check $\epsilon_\cF([\ell]) = \epsilon_\cF(\lambda)\epsilon_\cF([e])$. By definition,
$$\epsilon_\cF([\ell]) = \textrm{tr}(\rho(\ell)- \rho(m\ell)) = \textrm{tr}(\rho(\ell)(\textrm{id}_V-\rho(m))),$$
and 
$$\epsilon_\cF(\lambda)\epsilon_\cF([e]) = (\textrm{tr}(\rho(\ell))-\textrm{tr}(\rho'(K)))\textrm{tr}(\textrm{id}_V-\rho(m)).$$
By Remark \ref{KCHRmk}, and that $\cF = j_*\cE$, we have
$$\rho(m) = \begin{pmatrix} \mu_0 & \\ & I_{n-1} \end{pmatrix},
\quad
\rho(\ell) = \begin{pmatrix} \lambda_0 & \\ &\ast_{n-1} \end{pmatrix},
\quad \rho'(K) = \begin{pmatrix} \ast_{n-1} \end{pmatrix}.
$$
Both hand sides equal to $\lambda_0(1-\mu_0)$.

If $M = E_{12}$, then we use the fact that $\ell$ and $m$ commute to obtain
$$\rho(m) = \begin{pmatrix} 1 & 1 & \\ & 1&  \\ & & I_{n-2} \end{pmatrix},
\quad
\rho(\ell) = \begin{pmatrix} a & b&  \\ & a& \\ & &\ast_{n-2} \end{pmatrix},
\quad \rho'(K) = \begin{pmatrix} a & \\ &\ast_{n-2} \end{pmatrix}.
$$
Note that if $\{v_1,\dotsb, v_n\}$ is the basis for $V$, then $W\subset V$ has the basis $\{v_1, \hat{v}_2, v_3,\dotsb,v_n\}$. Let $M : = \textrm{id}_V-\rho(m)$, $C:= \rho(\gamma)$ for a choice $\gamma\in\pi_K$. Then
$$\epsilon_\cF([\gamma\ell]) = \textrm{tr}(\rho(\gamma\ell)- \rho(m \gamma \ell)) = \textrm{tr} (MC\rho(\ell)) = -aC_{21},$$
and
\begin{align*}
\epsilon_\cF([\ell\gamma]) &= \textrm{tr}(\rho(\ell\gamma)- \rho(m \ell \gamma )) \\
&= \textrm{tr}(\rho(\ell\gamma)- \rho(\ell m  \gamma )) \qquad [\textrm{Since } \ell \textrm{ and } m \textrm{ commute}] \\
&= \textrm{tr} (\rho(\ell)MC) = -aC_{21},
\end{align*}
and
$$\epsilon_\cF(\lambda)\epsilon_{\cF}([\gamma]) = \big(\textrm{tr} (\rho(\ell)) - \textrm{tr} (\rho'(K))\big) \textrm{tr}(MC) = -aC_{21}.$$

We have verified all the relations in all cases. The proof is complete.

\end{proof}

\begin{rmk}
St\'ephane Guillermou suggested the following improvement of the proof, on how to verify the skein relations.  The simpleness, by Definition \ref{simplesheaf}, yields for the first case,
$$0\rightarrow W\xrightarrow{T} V\xrightarrow{u} k\rightarrow 0.$$
Note that $\textrm{id}_V-\rho(m)\in \textrm{Hom}(V,V)$ is an endomorphism. Recall from Lemma \ref{SheafRep} (3b) that $m$ acts on $W$ as identity, hence $(\textrm{id}_V-\rho(m))\circ T =0$. It further implies that $\textrm{id}_V-\rho(m)\in \textrm{Hom}(V,V)$ factors through $k$. Namely, there is some morphism $a\in \textrm{Hom}(k,V)$, such that $\textrm{id}_V-\rho(m) = p\circ a$. Therefore $M:= \textrm{id}_V -\rho(m)$ has rank $1$, which leads to $\textrm{tr}(MAMB) = \textrm{tr}(MA)\textrm{tr}(MB)$.

For the second case when
$$0\rightarrow k \rightarrow W\rightarrow V\rightarrow 0,$$
we have $\textrm{id}_V-\rho(m) = 0$ by Lemma \ref{SheafRep} (3b). Then $\epsilon_\cF([\gamma]) = 0$. The skein relations follow.

\end{rmk}

Let $\cA ug$ be the set of augmentations. By sending $\cF$ to $\epsilon_\cF$, we have defined a map 
$$\cS\rightarrow \cA ug.$$
The map descends to the moduli set $\cM$, in the sense that we fix a representative in each isomorphism class, for example the representatives in Proposition \ref{Sheafuptoloc}. For other representatives in the same isomorphism class, sometimes there is a sign ambiguity, which we discuss in the following remark.

\begin{rmk}\label{modulisheaftoaugrmk}
Given a local system $\cL_X\in loc(X)$, one can check that the associated map $\epsilon_{\cL_X}$ is zero for all elements in $\pi_K$. For any $\gamma\in \pi_K$, if $\cF_1 \sim \cF_2$ by (ER1) and (ER3), then $\epsilon_{\cF_1}([\gamma]) = \epsilon_{\cF_2}([\gamma])$, and if they are related by (ER2), then $\epsilon_{\cF_1}([\gamma]) = -\epsilon_{\cF_2}([\gamma])$. For $\mu$ and $\lambda$, $\epsilon(\mu)$ is determined by $\epsilon([e])$ according to the normalization, $\epsilon(\lambda)$ remains the same under (ER1) and (ER3), and negates under (ER2). In fact, the only chance we need to apply (ER2) is the following short exact sequence, or its variation by adding a local system $\cL_X$,
$$0\rightarrow j_!k_{X\setminus K}\rightarrow k_X \rightarrow i_*k_{K}\rightarrow 0.$$
For the two simples sheaves here, $\epsilon([\gamma]) = 0$ for all $\gamma\in \pi_K$.
\end{rmk}

The sign can be better understood and taken care of if we work with simple sheaves in the derived cateogry. Suppose $V^\bullet$ is a chain complex and $T^\bullet: V^\bullet\rightarrow V^\bullet$ is a chain map. For each $n$, there is a linear endomorphism $H^nT^\bullet: H^nV^\bullet\rightarrow H^nV^\bullet$ of the $n$-th homology vector space. We define
$$\textrm{tr}^\bullet(T^\bullet) = \sum_{n\in\bbZ} (-1)^n\, \textrm{tr}(H^nT^\bullet).$$
In some contexts, $\textrm{tr}^{\bullet}$ is referred to as the supertrace. Recall from Section \ref{Sec:DerivedSheaves} that an object $\cF\in Sh_{\Lambda_K}(X)$ gives the following data: (1) a chain complex of $k[\pi_K]$-module $V^\bullet$ (the $\pi_K$-action given by $\rho$), (2) a chain complex of $k[\bbZ_K]$-module $W^\bullet$ (the $\bbZ_K$-action given by $\rho'$), and (3) a chain map $T^\bullet: W^\bullet\rightarrow V^\bullet$. The simpleness of $\cF$ requires $cone(T^\bullet) \cong k[d]$ for some integer $d$. We define $\epsilon_\cF$ to be:
\begin{align*}
\epsilon_\cF(\lambda) &= (-1)^d(\textrm{tr}^\bullet(\rho(\ell)) - \textrm{tr}^\bullet(\rho'(K))); \\
\epsilon_\cF(\mu) &= (-1)^d\,\textrm{tr}^\bullet(\rho(m) -\textrm{id}_{V^\bullet}) + 1; \\
\epsilon_\cF([\gamma]) &= (-1)^d\,\textrm{tr}^\bullet \left(\rho(\gamma) -\rho(m\gamma) \right).
\end{align*}
It is straightforward to check that shifting the degree of a sheaf does not change the associated map, namely $\epsilon_{\cF}= \epsilon_{\cF[1]}$. This makes sense. The framed cord algebra models the degree zero homology of the dga of the knot conormal. The definition of the dga depends on a choice of the Maslov potential, whose sheaf counterpart is the homological degree. Since the knot conormal is connected, different choices of Maslov potentials give identical dgas and augmentations, and so should different choices of homological degrees of sheaves.

\begin{prop}\label{Dsheaftoaug}
If $\cF\in \widehat{\cM}$, then $\epsilon_{\cF}$ is an augmentation.
\end{prop}
\begin{proof}
Suppose there is a distinguished triangle,
$$\cF_1\rightarrow \cF_2\rightarrow \cL_X\xrightarrow{+1},$$
then $\epsilon_{\cF_1} = \epsilon_{\cF_2}$. Hence it suffices to check for the representatives in each isomorphism classes. Because shifting the degree does not change the associated map, by Proposition \ref{DerivedSheafloc}, it reduces to the simple sheaves concentrated at degree zero. By Theorem \ref{sheaftoaug}, we see that $\epsilon_\cF$ is a well-defined augmentation.
\end{proof}

\begin{cor}
There is a well-defined map
$$\cM\rightarrow \cA ug,$$
by sending $\cF\mapsto \epsilon_\cF$.
\end{cor}
\begin{proof}
It follows from Proposition \ref{Dsheaftoaug}, Corollary \ref{moduliquotient}, and the fact that shifting the degree of a simple sheaf does not change the associated augmentation.
\end{proof}
\begin{rmk}
Responding to the sign ambiguity in the previous Remark \ref{modulisheaftoaugrmk}, the short exact sequence $0\rightarrow \cF_1 \rightarrow \cL_X\rightarrow \cF_2\rightarrow 0$ in (ER2) can be rewritten as a distinguished triangle
$$\cL_X\rightarrow \cF_2\rightarrow \cF_1[1]\xrightarrow{+1}.$$
Therefore we have $\epsilon_{\cF_2} =\epsilon_{\cF_1[1]} =\epsilon_{\cF_1}$. The sign ambiguity does not exist any more. 
\end{rmk}

\subsection{KCH representations and augmentations} \label{Sec:KCHAug}

Recall from Section \ref{Sec:KCHRep} that a KCH representation is a representation of the knot group $\rho: \pi_K\rightarrow GL(V)$ such that the meridian action is diagonalizable and equals to identity on a subspace of exact codimension $1$. Suppose $m$ is a meridian and $\ell$ is the longitude. An eigenvector of the eigenvalue $\mu_0$ in $\rho(m)$ is also an eigenvector of $\rho(\ell)$, corresponding to the eigenvalue $\lambda_0$.

Ng defined an augmentation $\epsilon$ from the KCH representation by assignments to the generators of the framed cord algebra
$$\epsilon_\rho(\mu) = \mu_0, \quad \epsilon_\rho(\lambda) =\lambda_0,\quad  \epsilon_\rho([\gamma]) = (1-\mu_0)\rho(\gamma)_{11},$$
where $\rho(\gamma)_{11}$ is the $(1,1)$-entry of the matrix $\rho(\gamma)$ \cite{Ng3}. This construction gives a map
\begin{equation}\label{Cornwellsur}
\{\textrm{KCH representation } (\rho,V)\} \rightarrow \{\textrm{Augmentation }\epsilon \,|\, \epsilon(\mu)\neq 1\}.
\end{equation}
Cornwell proves that this map (\ref{Cornwellsur}) is surjective.  Moreover, every such augmentation lifts to an irreducible KCH representation, unique up to isomorphism \cite{Cor13b}.

By Theorem \ref{Sheafclassify}, the KCH representations can be naturally identified as a subset of simple abelian sheaves microsupported along the knot conormal. In the last section, we defined a map from these sheaves to the augmentations. The following proposition shows that (\ref{Cornwellsur}) factors through these two maps.

\begin{prop}\label{factorthrough}
Let $(\rho,V)$ be a KCH representation and $\cE$ the associated KCH local system. Let $\cF = j_*\cE$ be a simple sheaf. Then $\epsilon_\rho = \epsilon_\cF$.
\end{prop}

\begin{proof}
The sheaf determines a representation $\rho': \bbZ_K \rightarrow W$ (Lemma \ref{SheafRep}). By construction, $W$ is identified as a subspace of $V$. By Remark \ref{KCHRmk}, we can choose a basis of $V$ such that
$$\rho(m) = \begin{pmatrix} \mu_0 & \\ & I_{n-1} \end{pmatrix},
\quad
\rho(\ell) = \begin{pmatrix} \lambda_0 & \\ &\ast_{n-1} \end{pmatrix},
\quad \rho'(K) = \begin{pmatrix} \ast_{n-1} \end{pmatrix}.
$$
Here, if $\{v_1,v_2, \dotsb, v_n\}$ is the basis for $V$, then $\{\hat{v}_1, v_2, \dotsb, v_n\}$ is the basis for $W$. Also note that $\ast_{n-1}$ in both $\rho(\ell)$ and $\rho'(K)$ refers to the same square matrix. It is straightforward to compute:
\begin{itemize}
\item
$\epsilon_\cF(\mu) = \textrm{tr}(\rho(m)) -\dim V + 1 = \mu_0 = \epsilon_\rho(\mu)$; and
\item
$\epsilon_\cF(\lambda) = \textrm{tr}(\rho(\ell)) - \textrm{tr}(\rho'(K)) = \lambda_0 = \epsilon_\rho(\lambda)$; and
\item
$\epsilon_\cF([\gamma]) = \textrm{tr}\left(\rho(\gamma) -\rho(m\gamma) \right) = \textrm{tr}( (1-\mu_0)E_{11}\rho(\gamma)) = (1-\mu_0)\rho(\gamma)_{11} = \epsilon_\rho([\gamma]).$
\end{itemize}
\end{proof}

\subsection{Unipotent KCH representations and augmentations}\label{Sec:UniKCHAug}
In this section, we present a correspondence between unipotent KCH representations and augmentations -- every augmentation with $\epsilon([e]) =0$ but $\epsilon ([\gamma])\neq 0$ for some $\gamma\in \pi_K$ can be lifted to  
an irreducible unipotent KCH representation, unique up to isomorphism.

Suppose $\{m_i\}_{i\in I}$ is a finite set of meridian generators of $\pi_K$. Suppose $I = \{1, \dotsb, N\}$ has size $N$. Unless $K$ is the unknot there is $N\geq 3$. Let $m= m_1$ be the preferred meridian, whose homology class is $\mu$. Since any meridian $m_t$ is conjugate to $m$, we choose $g_t\in \pi_K$ to be elements such that $m_t= g_t^{-1}m g_t$. 

Suppose $\epsilon: \cP_K\rightarrow k$ is an augmentation. We consider a square matrix $R$ defined by $\epsilon$, where entries are given by 
$$R_{ij} = \epsilon([g_ig_j^{-1}]).$$
Sometimes the matrix $R$ determines $\epsilon$, and sometimes one needs to specified $\epsilon(\lambda)$ in addition. The idea is that one can express a knot group element as a word of meridian generators, each of which is a conjugation of the preferred meridian $m$. Applying the skein relations iteratively, one obtains an expression without $m$, but only products of $g_ig_j^{-1}$. It becomes clear when $\epsilon(\lambda)$ needs to be specified after we prove the main Theorem \ref{CornNgfactor}. 

Let $R_j$ be the column vectors of $R$. We will construct a knot gorup representation over the following vector space $$V: = \textrm{Span}_k \{R_j\}_{j\in I}.$$

We adopt a convention using a floating index $\alpha$ exhausting $I$ to represent some column vectors of size $N$. For any $\gamma\in \pi_K$, define
$$\epsilon([g_\alpha \gamma]) := \big(\,\epsilon([g_1 \gamma])\, ,\, \epsilon([g_2 \gamma]) \, ,\, \dotsb \, ,\, \epsilon([g_N \gamma])\, \big).$$

\begin{prop}
The following map defines a representation $\rho: \pi_K\rightarrow GL(V)$:
\begin{equation}\label{LiftRep}
\rho(\gamma)R_j := \epsilon([g_\alpha \gamma g_j^{-1}]),\quad \textrm{for any }\gamma\in \pi_K.
\end{equation}
\end{prop}
\begin{rmk}
This to be justified representation is a simplification and an extension of the construction in \cite{Cor13b}. Cornwell's original construction introduced a localized algebra and a ``universal augmentation'', which work well but limited to the case when $\epsilon([e])\neq 0$. Our goal is to build the correspondence between augmentations and simple sheaves, in all three cases classified in Proposition \ref{DerivedSheafloc}. The new construction will adapt to all cases.
\end{rmk}

\begin{proof}
There are several things we need to justify. For any $\gamma\in \pi_K$, $\rho(\gamma)$ is closed and well-defined. The actions of based loops respect the group structure, namely the identity and the group product.

 (1) Closedness. One needs to check that $\rho(\gamma)R_j \in V$ for any $\gamma\in \pi_K$ and any $R_i$. It suffices to check for meridian generators. For any $R_j$ and any meridian generator $m_t$,
\begin{equation}\label{generatoraction}
\rho(m_t)R_j = R_j - \epsilon([g_t g_j^{-1}])R_t, \quad \rho(m_t^{-1})R_j =R_j + \epsilon([m^{-1}g_t g_j^{-1}])R_t,
\end{equation}
because
\begin{align*}
\rho(m_t)R_j &= \epsilon([g_{\alpha} m_t g_j^{-1}]) 	&& [ \textrm{Definition } (\ref{LiftRep})]\\
		&= \epsilon([g_{\alpha} g_t^{-1} m g_t g_j^{-1}]) && [m_t = g_t^{-1} m g_t]\\
		&= \epsilon([g_{\alpha} g_t^{-1}g_t g_j^{-1}]) - \epsilon([g_{\alpha} g_t^{-1}])\epsilon([g_t g_j^{-1}])	&& [\textrm {Skein relation}]\\
		&= R_j - \epsilon([g_t g_j^{-1}])R_t,
\end{align*}
and similarly
\begin{align*}
\rho(m_t^{-1})R_j 
		&= \epsilon([g_{\alpha} g_t^{-1} m^{-1} g_t g_j^{-1}]) \\
		&= \epsilon([g_{\alpha} g_t^{-1}mm^{-1}g_t g_j^{-1}]) + \epsilon([g_{\alpha} g_t^{-1}])\epsilon([m^{-1}g_t g_j^{-1}]) \\
		&= R_j + \epsilon([m^{-1}g_t g_j^{-1}])R_t.
\end{align*}

(2) Identity. It is straightforward to see that $\rho([e])R_j = \epsilon([g_\alpha e g_j^{-1}]) = \epsilon([g_\alpha g_j^{-1}]) = R_j$.

(3) Group product. We need $\rho(\gamma_1)\rho(\gamma_2)= \rho(\gamma_1\gamma_2)$, and we prove it by an induction on the word length of $\gamma_2$. To see the initial step, let $\gamma_2 = m_t$. For any $\gamma_1\in \pi_K$ and any $R_j$, there are
\begin{align*}
\rho(\gamma_1)\rho(m_t)R_j 
		&= \rho(\gamma_1)\big(R_j - \epsilon([g_t g_j^{-1}])R_t\big) && [\textrm{Equation }(\ref{generatoraction})]\\
		&= \epsilon([g_\alpha \gamma_1 g_j^{-1}]) - \epsilon([g_t g_j^{-1}])\epsilon([g_\alpha \gamma_1 g_t^{-1}]), && [ \textrm{Definition } (\ref{LiftRep})]
\end{align*}
and
\begin{align*}
\rho(\gamma_1 m_t)R_j 
		&= \epsilon([g_\alpha \gamma_1 m_t g_j^{-1}]) \\
		&= \epsilon([g_\alpha \gamma_1 g_t^{-1} m g_t g_j^{-1}]) \\
		&= \epsilon([g_\alpha \gamma_1 g_t^{-1}g_{t} g_j^{-1}]) - \epsilon([g_\alpha \gamma_1 g_t^{-1}]) \epsilon([g_t g_j^{-1}]).
\end{align*}
Therefore $\rho(\gamma_1)\rho(m_t) = \rho(\gamma m_t)$. Similarly there is $\rho(\gamma_1)\rho(m_t^{-1}) = \rho(\gamma_1m_t^{-1})$ because
$$\rho(\gamma_1)\rho(m_t^{-1})R_j = \epsilon([g_\alpha\gamma_1 g_j^{-1}]) + \epsilon([g_\alpha\gamma_1g_t^{-1}])\epsilon([m^{-1} g_t g_j^{-1}])= \rho(\gamma_1m_t^{-1})R_j.$$

To see the induction step, suppose $\gamma_2 = \gamma_2'm_t^{\pm1}$ such that $\gamma_2'$ satisfied the induction hypothesis that $\rho(\gamma_1)\rho(\gamma_2') = \rho(\gamma_1\gamma_2')$ for any $\gamma_1\in \pi_K$, then 
$$\rho(\gamma_1)\rho(\gamma_2'm_t^{\pm1}) = \rho(\gamma_1)\rho(\gamma_2')\rho(m_t^{\pm1}) = \rho(\gamma_1\gamma_2')\rho(m_t^{\pm1}) = \rho(\gamma_1\gamma_2' m_t^{\pm1}).$$
Therefore $\rho(\gamma_1)\rho(\gamma_2) = \rho(\gamma_1\gamma_2)$. 

(4) Well-definedness. We want to show that for any $\gamma\in \pi_K$, if there is a subset $I'\subset I$ such that $\sum_{i\in I'} a_iR_i = 0$, then $\sum_{i\in I'} a_i\rho(\gamma)R_i = 0$. It suffices to proof for $\gamma = m_t^{\pm 1}$, then the argument continues by induction on the word length of $\gamma$. Let $\gamma = m_t$, by (\ref{generatoraction}), we have
$$\sum_{i\in I'} a_i\rho(\gamma)R_i = \sum_{i\in I'} a_iR_i - \sum_{i\in I'} a_i \epsilon([g_tg_i^{-1}])R_t. $$
The first summand is zero by hypothesis. In the second summand, $\sum_{i\in I'} a_i \epsilon([g_tg_i^{-1}])$ is the $t$-th entry of $\sum_{i\in I'} a_iR_i =0$, and therefore also equals to zero. The argument for $\gamma= m_t^{-1}$ is similar. 

We complete the proof.
\end{proof}

Cornwell proved if $\epsilon([e]) = 0$ (or $\epsilon(\mu)=1$), then (\ref{LiftRep}) is an irreducible KCH representation \cite[Corollary 3.7]{Cor13b}. We show that if $\epsilon([e]) =0$ but $\epsilon([\gamma])\neq 0$ for some $\gamma\in \pi$, then (\ref{LiftRep}) defines an irreducible unipotent KCH representation. The next lemma unwraps the condition for the augmentation.

\begin{lem}\label{UniAugcond}
The following conditions are equivalent for an augmentation $\epsilon: \cP_K\rightarrow k$,
\begin{enumerate}
\item
$\epsilon([\gamma]) = 0$ for all $\gamma\in \pi_K$;
\item
$\epsilon([g_i^{-1}])= 0$ for all $i \in I$;
\item
$\epsilon([g_i])= 0$ for all $i \in I$.
\end{enumerate}
\end{lem}
\begin{proof}
By construction $g_1 =g_1^{-1} =e$, hence $\epsilon([e]) =0$ in either (2) or (3).

Obviously (1) $\Rightarrow$ (2) and (1) $\Rightarrow$ (3).

\smallskip
(2) $\Rightarrow$ (1). We prove by induction on the word length in terms of meridian generators.

Initial step. Suppose $m_t$ is a meridian generator, then
\begin{align*}
\epsilon([m_t]) 
		&= \epsilon([g_t^{-1} m g_t]) \\
		&= \epsilon([g_t^{-1}g_t])- \epsilon([g_t^{-1}])\epsilon([g_t]) &&[\textrm{Skein relation}] \\
		&= \epsilon([e]) - \epsilon([g_t^{-1}])\epsilon([g_t]) = 0,
\end{align*}
and
\begin{align*}
\epsilon([m_t^{-1}]) 
		&= \epsilon([g_t^{-1} m^{-1} g_t]) \\
		&= \epsilon([g_t^{-1}m m^{-1}g_t]) + \epsilon([g_t^{-1}])\epsilon([m^{-1}g_t]) \qquad[\textrm{Skein relation}]\\
		&= \epsilon([e]) + \epsilon([g_t^{-1}])\epsilon([m^{-1}g_t]) = 0.
\end{align*}

Induction step. Suppose $\epsilon([\gamma]) = 0$, we show that $\epsilon([m_t\gamma]) = \epsilon([m_t^{-1}\gamma])=0$ for any meridian generator $m_t$.
$$\epsilon([ m_t \gamma]) = \epsilon([ g_t^{-1} m g_t \gamma]) = \epsilon([ g_t^{-1}g_t \gamma])- \epsilon([ g_t^{-1}])\epsilon([g_t\gamma]) = \epsilon([\gamma]) - \epsilon([g_t^{-1}])\epsilon([g_t\gamma]) = 0,$$
and similarly,
\begin{align*}
\epsilon([ m_t^{-1} \gamma]) = \epsilon([ g_t^{-1} m^{-1} g_t \gamma]) 
	&= \epsilon([ g_t^{-1}m m^{-1} g_t \gamma]) + \epsilon([ g_t^{-1}])\epsilon([m^{-1}g_t\gamma]) \\
&= \epsilon([\gamma])+ \epsilon([g_t^{-1}])\epsilon([m^{-1}g_t\gamma]) = 0.
\end{align*}
We complete the induction, proving that (2) $\Rightarrow$ (1).

\smallskip
The proof of (3) $\Rightarrow$(1) is similar, except performing the induction on $\gamma m_t$ or $\gamma m_t^{-1}$. 
\end{proof}

\begin{prop}\label{IrrUniKCHLift}
If $\epsilon([e])=0$ but $\epsilon([\gamma])\neq 0$ for some $\gamma \in \pi_K$, then (\ref{LiftRep}) defines an irreducible unipotent KCH representation.
\end{prop}

\begin{proof}

(1) We first prove that under the hypothesis, the matrix $R_{ij} = \epsilon([g_ig_j^{-1}])$ has rank at least $2$. Then the representation $(\rho,V)$ defined in (\ref{LiftRep}) has dimension at least $2$.

Since $\epsilon([e]) = 0$, the diagonal entries of $R$ are $0$.

Recall that $m_1= m$, which implies $g_1 = e$. Therefore $R_{i1} = \epsilon([g_ig_1^{-1}]) =\epsilon([g_i])$ and $R_{1j} = \epsilon([g_1g_j^{-1}]) = \epsilon([g_j^{-1}])$. By Lemma \ref{UniAugcond}, neither the first column $R_{i1}$ nor the first row  $R_{1j}$ is zero, otherwise $\epsilon([\gamma]) = 0$ for all $\gamma \in \pi_K$, contradicting the hypothesis.

Since the first row is not zero, there is a non-zero entry, say $R_{1s}\neq 0$. Then the column vector $R_s$ is non-zero, and linearly independent from $R_1$ because $R_{11}=0$ but $R_{1s}\neq 0$. There are two linearly independent non-zero vectors $R_1$ and $R_s$. Hence the rank is at least $2$.

Finally we compute $\rho(m)$. By (\ref{generatoraction}), there is $\rho(m)R_j = R_j - \epsilon ([g_j^{-1}])R_1$. In particular, $\rho(m)R_1 = R_1$ because $\epsilon([g_1^{-1}]) = \epsilon([e]) = 0$. By Lemma \ref{StdMatrix} (which has a dimension constraint), $\rho(m) = \textrm{id}_V + E_{12}$ under some basis. It is a unipotent KCH representation by definition.

(2) By Lemma \ref{UniKCHirrsubrep} and Lemma \ref{quottrivial2}, every unipotent KCH representation $(\tilde{\rho}, \tilde{V})$ contains a unique irreducible unipotent KCH sub-representation, characterized by $\textrm{Span}_k\{\tilde{V}_I\}$ where $\tilde{V}_i:= \textrm{im}\, (\textrm{id}_{\tilde{V}} - \tilde\rho(m_i))$ for each $i\in I$. Equation (\ref{generatoraction}) yields $\rho(m_i)R_j = R_ j -\epsilon([g_ig_j^{-1}])R_i$. In other words, if we fix $i\in I$, then for any other $j\in I$, there is
$$(\textrm{id}_V - \rho(m_i))R_j = \epsilon([g_ig_j^{-1}])R_i.$$
Therefore $V_i := \textrm{im}\,(\textrm{id}_V - \rho(m_i)) = \textrm{Span}_k\{R_i\}$. By definition $V = \textrm{Span}_k\{V_I\}$, which is thusly irreducible.
\end{proof}

Next we show that the lifted unipotent KCH representation in turn induces the augmentation begun with.
\begin{prop}\label{UniKCHLiftMatch}
Suppose $\cE_u$ is a unipotent KCH representation defined by an augmentation $\epsilon$ as in (\ref{LiftRep}). Let $\cF = j_*\cE_u$ be the associated simple sheaf. Then $\epsilon_\cF = \epsilon.$
\end{prop}
\begin{proof}
Because both $\epsilon_\cF$ and $\epsilon$ are homomorphisms from $\cP_K$ to $k$, it suffices to check for the generators of $\cP_K$, namely $\mu, \lambda$ and $[\gamma]$ for $\gamma\in\pi_K$. 

In fact, only $\epsilon_\cF ([\gamma]) = \epsilon([\gamma])$ is necessary. The augmented values of $\mu$ and $\lambda$ automatically agree by the following argument. Since $\cE_u$ is a unipotent KCH representation, there is $\epsilon_\cF(\mu) = 1 = \epsilon(\mu)$. By the hypothesis on $\epsilon$, there exists $\gamma\in \pi_K$ such that $\epsilon([\gamma])\neq 0$. In Definition \ref{AugDef}, the longitude relation yields $\epsilon([\ell \gamma]) = \epsilon(\lambda)\epsilon([\gamma])$, which further implies $ \epsilon(\lambda) = \epsilon([\gamma])^{-1}\epsilon([\ell \gamma])$. The same holds for $\epsilon_\cF$. Therefore, if we have verified that $\epsilon_\cF([\gamma]) = \epsilon([\gamma])$ for all $\gamma\in \pi_K$, then $\epsilon_\cF(\lambda) = \epsilon_\cF([\gamma])^{-1}\epsilon_\cF([\ell \gamma]) = \epsilon([\gamma])^{-1}\epsilon([\ell \gamma]) =\epsilon(\lambda)$.

We prove $\epsilon_\cF ([\gamma]) = \epsilon([\gamma])$ by induction on the word length of meridian generators. To prepare the proof, suppose $I':=\{j_1,\dotsb, j_k\}\subset \{2,\dotsb,N\}\subset I$ is a subset of indices, such that $\{R_1,R_{j_1}, \dotsb,R_{j_k}\}$ is a basis of $V$. Since $\rho(m)R_j = R_j - \epsilon ([g_i^{-1}])R_1$ by (\ref{generatoraction}), under the chosen basis there is
$$
\rho(m)=
\begin{pmatrix}
1 & -\epsilon([g_{j_1}^{-1}]) & \dotsb &  -\epsilon([g_{j_k}^{-1}]) \\
 & 1 &  &  \\
 &  & \ddots &  \\
 &  &  & 1 
\end{pmatrix}.
$$
Then for any $\gamma \in \pi_K$, there is
\begin{equation}\label{coefficient}
\epsilon_\cF([\gamma]) = \textrm{tr}\big((\textrm{id}_V-\rho(m))\rho(\gamma)\big) = \sum_{\ast =1}^k\epsilon([g_{j_\ast}^{-1}])\rho(\gamma)_{\ast 1},
\end{equation}
where $\rho(\gamma)_{\ast 1}$ is the $(\ast+1, 1)$ entry of the matrix $\rho(\gamma)$ under the chosen basis (Note this unusual convention records entries in the first column starting from the second row).

\smallskip
Initial step. We need to check for $\gamma = m_t^{\pm1}$ for any meridian generator $m_t$. 

Consider $\gamma = m_t$ in the following cases. If $m_t =m$, it is straightforward to compute $\epsilon_\cF([m]) = 0$. Meanwhile by the meridian relation in Definition $\ref{AugDef}$, $\epsilon([m]) = \epsilon([me]) = \epsilon(\mu[e]) = 0$, which equals to $\epsilon_\cF([m])$. Next we assume $t \in I' =\{j_1,\dotsb, j_k\}$, then by (\ref{generatoraction}) there is
\begin{equation}\label{mtaction}
\rho(m_t)R_1 = R_1 - \epsilon([g_t g_1^{-1}])R_t = R_1 - \epsilon([g_t])R_t.
\end{equation}
We have $\rho(\gamma)_{\ast 1} = - \epsilon([g_t])$ when $j_\ast = t$, and $\rho(\gamma)_{\ast 1} = 0$ when $j_\ast \in I'\setminus \{t\}$. By equation (\ref{coefficient}) we have
$$\epsilon_\cF([m_t]) = -\epsilon ([g_t^{-1}])  \epsilon([g_t]) = - \epsilon([e]) + \epsilon([g_t^{-1}mg_t]) = \epsilon([m_t]).$$
Finally if $t\notin I'$, assume a linear combination
$R_t = cR_1 + \sum_{\ast = 1}^{k}c_\ast R_{j_\ast}.$
In particular, it implies $\epsilon([g_t^{-1}]) = \sum _{\ast =1}^k c_\ast \epsilon([g_{j_*}^{-1}])$ by considering the first entry. Applying equation (\ref{mtaction}) and using the linear combination, we get
\begin{equation}\label{generalmtaction}
\rho(m_t)R_1 = R_1 - \epsilon([g_t]) R_t = (1-c)R_1 -\sum_{\ast =1}^kc_\ast\epsilon([g_t])R_{j_\ast}.
\end{equation}
Then there is
$$\epsilon_{\cF}([m_t]) = -\sum_{\ast=1}^{k} \epsilon([g_{j_\ast}^{-1}]) c_\ast\epsilon([g_t]) = -\epsilon([g_t]) \epsilon\left(\sum_{\ast=1}^{k} c_\ast [g_{j_\ast}^{-1}]\right) = - \epsilon([g_t]) \epsilon ([g_t^{-1}]) = \epsilon([m_t]).$$
Here the first equality is due to equations (\ref{coefficient}) and (\ref{generalmtaction}), the second equality is because $\epsilon$ is linear with respect to the scalar multiplication, the third equality is because of $\epsilon([g_t^{-1}]) = \sum _{\ast =1}^k c_\ast \epsilon([g_{j_*}^{-1}])$ derived from the linear combination.

We have proven for $\gamma = m_t$ for any meridian generator $m_t$, and will argue for $\gamma = m_t^{-1}$. Observe that for any augmentation $\epsilon_0$ with $\epsilon_0(\mu) = 1$, there is $\epsilon_0([m\gamma]) = \epsilon_0([\gamma]) = \epsilon_0([m^{-1}\gamma])$ for any $\gamma\in \pi_K$ (by the meridian relation in Definition \ref{AugDef}). Consequently by the skein relations, there are 
$$\epsilon_0([m_t]) = \epsilon_0([g_t^{-1}m g_t]) = \epsilon_0([e])+ \epsilon_0([g_t^{-1}])\epsilon_0([g_t]) = \epsilon_0([g_t^{-1}])\epsilon_0([g_t]),$$
and
$$\epsilon_0([m_t^{-1}]) = \epsilon_0([g_t^{-1}m^{-1} g_t]) = \epsilon_0([e]) - \epsilon_0([g_t^{-1}])\epsilon_0([m^{-1}g_t]) = -\epsilon_0([g_t^{-1}])\epsilon_0([g_t]).$$
Therefore $\epsilon_0([m_t]) = -\epsilon_0([m_t^{-1}]) $. Further we have
$$\epsilon_\cF([m_t^{-1}]) = -\epsilon_\cF([m_t]) = -\epsilon([m_t]) = \epsilon([m_t^{-1}]),$$
completing the initial step.

Induction step. Suppose $\epsilon_\cF([\gamma]) = \epsilon([\gamma])$, we show that $\epsilon_\cF([m_t \gamma]) = \epsilon([m_t \gamma])$ for any generating meridian $m_t$. Compute
\begin{align*}
\rho(m_t\gamma)R_1 &= \epsilon([g_\alpha g_t^{-1}mg_t\gamma]) \\
	&= \epsilon([g_\alpha\gamma]) - \epsilon([g_\alpha g_t^{-1}])\epsilon([g_t \gamma])  \\ 
	&= \rho(\gamma)R_1 - \epsilon([g_t\gamma])R_t.
\end{align*}
Reorganizing the terms, we have
\begin{equation}\label{recursive}
(\rho(m_t\gamma)-\rho(\gamma))R_1 = -\epsilon([g_t\gamma])R_t.
\end{equation}
We continue the argument depending on whether $t\in I'$, as in the initial step. If so, by equations (\ref{coefficient}) and (\ref{recursive}), there is
$$\epsilon_\cF([m_t\gamma]) - \epsilon_\cF([\gamma]) = -\epsilon([g_t^{-1}])\epsilon([g_t\gamma]) = \epsilon([g_t^{-1}mg_t\gamma]) - \epsilon([\gamma]) = \epsilon([m_t\gamma]) - \epsilon([\gamma]).$$
Because $\epsilon_\cF([\gamma])=\epsilon([\gamma])$ by the induction hypothesis, there is $\epsilon_\cF([m_t\gamma])=\epsilon([m_t\gamma])$. If $t\notin I'$, we can again write down the linear combination, and then a similar argument proceeds.

Under the same induction hypothesis, it can be proven similarly that $\epsilon_\cF([m_t^{-1}\gamma]) = \epsilon([m_t^{-1}\gamma])$, except that the recursive formula (\ref{recursive}) appears slightly different:
$$(\rho([m_t^{-1}\gamma])-\rho(\gamma))R_1 = \epsilon([m^{-1}g_t\gamma])R_t.$$

We complete the induction argument, as well as the proof.
\end{proof}

Finally we present a uniqueness result.
\begin{prop}\label{uniqueness}
The irreducible unipotent KCH representation that induces a given augmentation $\epsilon$ with $\epsilon([e])= 0$ but $\epsilon([\gamma])\neq 0$ for some $\gamma\in \pi_K$, is unique up to isomorphism.
\end{prop}
\begin{proof}
Suppose $(\tilde\rho, \tilde{V})$ is an irreducible unipotent KCH representation which induces the augmentation $\epsilon$. Let $(\rho, V)$ be the unipotent KCH representation lifted from $\epsilon$ as in (\ref{LiftRep}). By Proposition \ref{IrrUniKCHLift}, it is also irreducible. It suffices to prove that $(\tilde\rho, \tilde{V})$ and $(\rho, V)$ are isomorphic. 

By Lemma \ref{UniKCHirrsubrep}, $\tilde{V} = \textrm{Span}_k\{\tilde{V}_I\}$, where $\tilde{V}_i := \textrm{im}\, (\textrm{id}_{\tilde{V}} - \tilde{\rho}(m_i))$ for each $i\in I$. A similar statement holds for $(\rho, V)$, and in addition there is $V_i = \textrm{Span}_k\{R_i\}$ (see the proof of Proposition \ref{IrrUniKCHLift}, part (2)). Let $\tilde{v}_1\in \tilde{V}_1$ be a non-zero vector. We will construct a representation morphism $\phi: (\tilde{\rho}, \tilde{V})\rightarrow (\rho, V)$.

Define $\phi(\tilde{v}_1) = R_1$, we will check that it extends to a representation morphism. For any $t\in I$, there is $m_t = g_t^{-1}mg_t$. We define $\tilde{v}_t := \tilde{\rho}(g_t^{-1})\tilde{v}_1$, and it spans $\tilde{V}_t$ because
$$\tilde{V}_t = \textrm{im}\, (\textrm{id}_{\tilde{V}} - \tilde{\rho}(m_t)) = \textrm{im}\, (\tilde{\rho}(g_t^{-1})(\textrm{id}_{\tilde{V}} - \tilde{\rho}(m))\tilde{\rho}(g_t)) = \textrm{Span}_k \{\tilde{\rho}(g_t^{-1})\tilde{v}_1\} = \textrm{Span}_k \{\tilde{v}_t\}.$$ 
Using definition (\ref{LiftRep}), we can compute that
$$R_t = \epsilon([g_\alpha g_t^{-1}]) = \epsilon([g_\alpha g_t^{-1}g_1^{-1}])= \rho(g_t^{-1})R_1.$$
Comparing these calculations, if we set $\phi(\tilde{v}_t) = R_t$ for all $t\in I$, then $\phi$ respects the group actions on the two vector spaces spanned by $\{\tilde{v}_i\}_{i\in I}$ and $\{R_i\}_{i\in I}$. Since $\{\tilde{v}_i\}_{i\in I}$ (resp. $\{R_i\}_{i\in I}$) is a spanning set of $\tilde{V}$ (resp. $V$), $\phi$ extends to a representation morphism from $\tilde{V}$ to $V$.

Because both $(\tilde{\rho}, \tilde{V})$ and $(\rho, V)$ are irreducible representations, and because $\phi$ is not zero, there is a representation isomorphism $(\tilde{\rho}, \tilde{V})\cong(\rho, V)$. We prove the uniqueness.

\end{proof}

\subsection{The sheaf-augmentation correspondence}\label{Sec:SheafAug}
Now we are ready to present the relation between sheaves and augmentations. It also becomes evident how the KCH representations come into the picture.

Recall that $\cS$ is the set of isomorphism classes of objects in $Mod^s_{\Lambda_K}(X)$, and $\cM$ is the moduli set of simple sheaves up to local systems (Section \ref{Sec:Moduli}). Also recall that $\cA ug$ is the set of augmentations. 

Let $\cK ch$ be the set of isomorphism classes of KCH representations, where the isomorphism is the representation isomoprhism. Let $\cK ch^{irr}\subset \cK ch$ be the subset of irreducible KCH representations.

\begin{thm}\label{CornNgfactor}
(1) The map from KCH representations to augmentations (\ref{Cornwellsur}) factors through the following diagram
$$\cK ch \hookrightarrow \cS \twoheadrightarrow \cA ug.$$

(2) It further induces the following diagram
$$\cK ch^{irr} \hookrightarrow \cM \xrightarrow{\sim} \cA ug.$$

Moreover, the isomorphism $\cM\cong \cA ug$ is summarized in the following table.
\smallskip
\begin{center}
\begin{tabular}{| r | l |}
\hline
\rule{0pt}{2.3ex}
$\cM$ & $\cA ug$ \\[0.045cm]
\hline
\rule{0pt}{2.3ex}
$\cE$ irreducible KCH, $\cF = j_*\cE$ &  $\epsilon([e])\neq 0$ \\
\; \, $\cE_u$ irreducible unipotent KCH, $\cF = j_*\cE_u$ & $\epsilon([e])= 0$ but $\epsilon([\gamma])\neq 0$ for some $\gamma\in \pi_K$ \;\\
$\cG_\alpha$ rank $1$ on the knot, $\cF = i_*\cG_\alpha[-1]$ & $\epsilon([\gamma])= 0$ for all $\gamma\in \pi_K$ \\[0.045cm]
\hline
\end{tabular}
\end{center}
\end{thm}

\begin{proof}
(1) The first map follows from the classification Theorem \ref{Sheafclassify}, as well as the injectivity. The second map follows from Theorem \ref{sheaftoaug}. The composition giving (\ref{Cornwellsur}) is a consequence of Proposition \ref{factorthrough}.

It remains to show that the second map is surjective. There are three possibilities:
\begin{itemize}
\item
If an augmentation $\epsilon$ satisfies $\epsilon([e]) \neq 0$ (or equivalently $\epsilon(\mu) \neq 1$), then it arises from a KCH representations by \cite[Theorem 1.2]{Cor13b}. Suppose $\cE$ is the corresponding local system, set $\cF = j_*\cE$.
\item
If $\epsilon([e]) = 0$ but $\epsilon([\gamma])\neq 0$ for some $\gamma\in \pi_K$, then it lifts to a unipotent KCH representation by Proposition \ref{IrrUniKCHLift} and Proposition \ref{UniKCHLiftMatch}. Suppose $\cE_u$ is the corresponding local system, set $\cF = j_*\cE_u$.
\item
If $\epsilon([\gamma])=0$ for all $\gamma\in \pi_K$. In this case $\epsilon(\mu)=\epsilon([e])+1 =1$. Therefore the augmentation only depends on $\epsilon(\lambda)$, where $\lambda$ is the class of the longitude. Then set $\cF = i_*\cG_\alpha[-1]$, where $\cG_\alpha$ is a rank $1$ local system whose monodromy is $\alpha = \epsilon(\lambda)$. \end{itemize} 
Comparing with the classification Theorem \ref{Sheafclassify}, we prove the surjectivity.

(2) Classes in $\cM$ are listed in Proposition \ref{Sheafuptoloc}. The uniqueness of the lifted KCH representation is proven in \cite[Theorem 1.2]{Cor13b}. The uniqueness of the lifted unipotent KCH representation is proven in Proposition \ref{uniqueness}. The uniqueness in the third case comes from the bijection between $\epsilon(\lambda)$ and $\alpha$.
\end{proof}

\begin{rmk}
Even though shifting the homological degree of a simple sheaf does not change the associated augmentation, we still make a degree shift in the third case so that $cone(T^\bullet)$ in all three cases are consistent.
\end{rmk}

\subsection{Augmentation polynomial}\label{Sec:AugPoly}

In this section we take $k= \bbC$. The augmentation polynomial is also a knot invariant. We first introduce the \textit{augmentation variety} \cite{Ng5}:
$$V_K = \{(\epsilon(\lambda),\epsilon(\mu))\in (\bbC^*)^2 \,|\, \epsilon \textrm{ is an augmentation} \}.$$
When the maximal-dimension part of the Zariski closure of $V_K$ is a codimension $1$ subvariety of $(\bbC^*)^2$, this variety is the vanishing set of a reduced polynomial (no repeated factor) $Aug_K(\lambda^{\pm 1},\mu^{\pm1})$, the augmentation polynomial of $K$. We can choose $Aug_K(\lambda,\mu)\in \bbZ[\lambda,\mu]$ with coprime coefficients, which is then well-defined up to an overall sign.

Recall from (\ref{extension}) that $\cF\in Mod^s_{\Lambda_K}(X)$ can be represented by a short exact sequence:
\begin{equation}\label{Ext}
0\rightarrow j_!\cH \rightarrow \cF \rightarrow i_*\cG \rightarrow 0,
\end{equation}
where $\cH \in loc (X\setminus K)$ and $\cG \in loc(K)$, determining a class in $\textrm{Ext}^1_X(i_*\cG, j_!\cH)$.

Let $\cM_0\subset \cM$ be the subset consisting representatives that induce trivial extension classes in $\textrm{Ext}^1_X(i_*\cG, j_!\cH)$.

\begin{prop}
If $\cF\in \cM_0$, then the induced augmentation $\epsilon_\cF$ determines a point in $\{(1-\lambda)(1- \mu)=0\}\subset (\bbC^*)^2$. Moreover, the map from $\cM_0$ to $\{(1-\lambda)(1- \mu)=0\}\subset (\bbC^*)^2$ is bijective.

\end{prop}
\begin{proof}

An augmentation $\epsilon$ which determines a point in $\{(1-\lambda)(1- \mu)=0\}\subset (\bbC^*)^2$ satisfies either $\epsilon(\lambda) =1$ or $\epsilon(\mu)=1$.

The extension (\ref{Ext}) being trivial means $\cF \cong j_!\cH \oplus i_*\cG$. Hence the linear transform $T: W\rightarrow V$ defined in Lemma \ref{SheafRep} is a zero map. The simpleness further requires either $V =\bbC, W=0$ or $W=\bbC, V=0$. 

If $V= \bbC, W=0$, then $\rho:\pi_K\rightarrow GL(V)$ is a one dimensional representation. Since an irreducible KCH representation has to be at least dimension $2$. Therefore by Proposition \ref{Sheafuptoloc}, $\cF = j_*\cE = j_!\cE$, where $\cE$ is a rank $1$ KCH local system. The underlining representation is abelian, and hence factors through $H_1(X\setminus K) =\bbZ$. Therefore $\epsilon_\cF(\mu) = \mu_0$, $\epsilon_\cF(\lambda) = 1$ and $\epsilon_\cF([\gamma]) = (1-\mu_0)\mu_0^{\textrm{lk}(K, \gamma)}$ for $\gamma\in \pi_K$, where $\textrm{lk}$ is the linking number. This case gives $\{\mu\neq 1, \lambda=1\}\subset (\bbC^*)^2$.

If $W =\bbC, V=0$, then $\cF = i_*\cG_\alpha$ for a rank $1$ local system $\cG_\alpha$ supported on the knot, according to Proposition \ref{Sheafuptoloc}. In this case $\epsilon_\cF(\mu) = 1$, $\epsilon_\cF(\lambda) = -\alpha$, and  $\epsilon_\cF([\gamma])=0$ for all $\gamma\in \pi_K$. This case gives $\{\mu = 1\}\subset (\bbC^*)^2$.

Overall, the map from $\cM_0$ to $\{(1-\lambda)(1-\mu)=0\}\subset (\bbC^*)^2$ is bijective.
\end{proof}

We reprove the following result in \cite{Ng3}.

\begin{cor}\label{universalfactor} For any knot $K$,  $(1-\lambda)(1-\mu) \,|\, Aug_K(\lambda, \mu).$
\end{cor}

\begin{rmk}
We interpret the augmentations in $(1-\lambda)(1-\mu)$ as sheaves coming from trivial extensions. From the sheaf perspective, the extension class $Ext^1_X(i_*\cG, j_!\cH)$ depends only on $\cG$, and $\cH$ restricted to a neighborhood of the knot. Since the neighborhood of any knot is the same as that of the unknot, the augmentation polynomial should be divisible by $(1-\lambda)(1-\mu)$, the augmentation polynomial of the unknot.
\end{rmk}

\end{document}